\documentclass[reqno, 11pt]{amsart}
\usepackage{geometry}
\usepackage{units}
\usepackage{textcomp,mathrsfs}
\usepackage{amstext}
\usepackage{amsthm}
\usepackage{amssymb}
\usepackage{booktabs}

\makeatletter
\numberwithin{equation}{section}
\numberwithin{figure}{section}
\numberwithin{table}{section}

\newtheorem{maintheorem}{Theorem}
\newtheorem{maincoro}[maintheorem]{Corollary}

\newtheorem{theorem}{Theorem}[section]
\newtheorem*{theorem*}{Theorem}

\newtheorem{proposition}[theorem]{Proposition}
\newtheorem{observation}[theorem]{Observation}

\theoremstyle{definition}{
\newtheorem{example}[theorem]{Example}
\newtheorem{definition}[theorem]{Definition}
\newtheorem*{definition*}{Definition}

\newtheorem*{example*}{Example}
\newtheorem{remark}[theorem]{Remark}
\newtheorem*{remarks*}{Remarks}
\newtheorem*{remark*}{Remark}

}

\usepackage{bbm}
\usepackage[framemethod=tikz]{mdframed}
\usepackage{amsmath,amsfonts,amssymb,graphicx,amsthm,enumerate,url,tikz,bbm}
\usepackage[noadjust]{cite}
\usepackage{stmaryrd,paralist}
\usepackage{xifthen,xcolor,tikz,setspace,caption}
\usepackage[noadjust]{cite}
\usepackage[color,final]{showkeys} 

\colorlet{refkey}{orange!20}
\colorlet{labelkey}{blue!30}
\usepackage[colorlinks=true, pdfstartview=FitV, linkcolor=blue,
  citecolor=blue, urlcolor=blue,pagebackref=false]{hyperref}
\numberwithin{equation}{section}

\providecommand{\C}{\mathbb C}
\renewcommand{\C}{\mathbb C}
\newcommand{\R}{\mathbb R}

\newcommand{\Z}{\mathbb Z}

\newcommand{\col}{{\operatorname{col}}}

\renewcommand{\P}{\mathbb{P}}

\newcommand{\Mod}[1]{\,\left(\textup{mod}\;#1\right)}
\DeclareMathOperator{\PGL}{PGL}
\DeclareMathOperator{\PSL}{PSL}

\renewcommand{\epsilon}{\varepsilon}

\newcommand{\cG}{{\mathcal{G}}}

\newcommand{\cW}{{\mathcal{W}}}

\newcommand{\cY}{{\mathcal{Y}}}
\newcommand{\cB}{{\mathcal{B}}}

\newcommand{\fC}{{\mathfrak{C}}}
\newcommand{\fX}{{\mathfrak{X}}}

\newcommand{\one}{\mathbbm{1}}

\newcommand{\tmix}{t_\textsc{mix}}

\newcommand{\tv}{{\textsc{tv}}}

\newcommand{\diam}{\operatorname{diam}}
\newcommand{\diag}{\operatorname{diag}}

\newcommand{\Lp}[1][p]{\mbox{\tiny $(L^p)$}}
\DeclareMathOperator{\rank}{rank}

\DeclareMathOperator{\Pspec}{PSpec}

\author[E. Lubetzky]{E. Lubetzky}
\address{E.\ Lubetzky\hfill\break
Courant Institute of Mathematical Sciences\hfill\break
New York University\\ 251 Mercer St.\\ New York, NY~10012\\~USA.}
\email{eyal@courant.nyu.edu}

\author[A. Lubotzky]{A. Lubotzky}
\address{A.\ Lubotzky\hfill\break
Einstein Institute of Mathematics\hfill\break 
Hebrew University\\ Jerusalem~91904\\ Israel.}
\email{alex.lubotzky@mail.huji.ac.il}

\author[O. Parzanchevski]{O. Parzanchevski}
\address{O.\ Parzanchevski\hfill\break
Einstein Institute of Mathematics\hfill\break 
 Hebrew University\\ Jerusalem~91904\\ Israel.}
\email{parzan@mail.huji.ac.il}

\makeatother

\begin{document}

\title[Random walks on Ramanujan complexes and digraphs]{Random walks on\\
 Ramanujan complexes and digraphs}

\begin{abstract}
The cutoff phenomenon was recently confirmed for random walks on Ramanujan graphs by the first author and Peres. In this work, we obtain analogs in higher dimensions, for random walk operators on any Ramanujan complex associated with a simple group $G$ over a local field $F$. We show that if $T$ is any $k$-regular $G$-equivariant operator on the Bruhat--Tits building with a simple combinatorial property (collision-free),  the associated random walk on the $n$-vertex Ramanujan complex has cutoff at time $\log_k n$. 
The high dimensional case, unlike that of graphs, requires tools from non-commutative harmonic analysis and the infinite-dimensional representation theory of $G$. Via these, we  show that operators $T$ as above on Ramanujan complexes give rise to Ramanujan \emph{digraphs} with a special property ($r$-normal), implying cutoff. Applications include geodesic flow operators, geometric implications, and a confirmation of the Riemann Hypothesis for the associated zeta functions over every group $G$, previously known for groups of type $\widetilde A_n$ and $\widetilde C_2$.
\end{abstract}

{\mbox{}
\vspace{-1.8cm}
\maketitle
}
\vspace{-0.71cm}

\section{Introduction}

A finite connected $k$-regular graph $\cG$ is Ramanujan (as defined in~\cite{LPS88}) if every eigenvalue $\lambda\in\R$ of its adjacency matrix satisfies either $|\lambda| \leq 2\sqrt{k-1}$
or $|\lambda| = k$.
In~\cite{Lubetzky2016}, the cutoff phenomenon was verified for simple random walk on  such graphs: the total variation distance of the  walk from its stationary  distribution exhibits a sharp transition at time $\frac{k}{k-2}\log_{k-1} n$, whereby it drops from near the maximum to near zero along an interval of $o(\log n)$ steps. One of the consequences of that analysis was that the typical distance between vertices in $\cG$ is $(1+o(1))\log_{k-1} n$. Our goal is to establish  analogous results in higher dimensions, for a broad class of random walk operators on Ramanujan complexes, as defined next.

Let $\mathbf{G}$ be a simple algebraic group of rank $d$, 
defined over a non-Archimedean local field $F$,
with residue field of order $q$. Let $G=\mathbf{G}\left(F\right)$, and let $\mathcal{B}=\mathcal{B}\left(G\right)$
be the associated Bruhat--Tits building, which is a $d$-dimensional contractible
simplicial complex.
Fix a $d$-dimensional cell $\sigma_{0}$ of $\mathcal{B}$ (a ``chamber'')
and let $\mathcal{I}$ denote the Iwahori subgroup of $G$, which is the (point-wise) stabilizer of $\sigma_{0}$
in $G$, uniquely determined up to conjugation (as $G$ acts transitively
on the chambers).
Let $\Gamma$ be a torsion-free cocompact discrete subgroup of
$G$. The quotient $X=\Gamma\backslash\mathcal{B}$ is a finite $d$-dimensional
simplicial complex, and we say it is a \emph{Ramanujan complex}\footnote{The definition given here slightly differs from the one first given in~\cite{Lubotzky2005a};
see the discussion in~\S\ref{sec:Ramanujan-complexes}.}~iff
\begin{equation}\begin{tabular}{l}every irreducible $\mathcal{I}$-spherical infinite-dimensional \\
$G$-subrepresentation of $L^2\left(\Gamma\backslash G\right)$ is tempered\end{tabular}\,,\label{eq:ram-def}
\end{equation}
where a representation is \emph{tempered} if it is weakly-contained in $L^{2}\left(G\right)$, and
it is \emph{$\mathcal{I}$-spherical} if it contains an
$\mathcal{I}$-fixed vector.

A \emph{$k$-branching} operator $T$ on a subset of cells $\fC\subset \cB$ is a map from $\fC$ to $\binom{\fC}k$, the $k$-element subsets of $\fC$; one can identify $T$ with the adjacency matrix of $\cY_{T,\fC}$, the $k$-out-regular digraph on the vertex set $\fC$, where  $(x,y)$ is an edge iff $y\in T(x)$. When $\cY_{T,\fC}$ is also $k$-in-regular, we say that $T$ is a \emph{$k$-regular branching} operator. We focus on $\fC$ that is $G$-invariant, and furthermore $T$ that is $G$-equivariant (that is, $T(gx) = gT(x)$ for all $x\in\fC$ and $g\in G$). In this case, the $k$-branching operator $T$ on $\fC\subset \cB$ induces such an operator on the finite subset of cells $\fX=\Gamma\backslash\fC \subset X$. 
We will study the mixing time of the random walk on $\fX$ associated with $T$.

Let $\|\mu-\nu\|_{\tv} =\sup_{A} [\mu(A)-\nu(A)]=\frac12\|\mu-\nu\|_{L^1}$ denote total-variation distance on a countable state space. The
 $L^1$ mixing time of a finite ergodic Markov chain with transition kernel $P$, stationary distribution $\pi$ and a worst initial state, is
\[ \tmix(\epsilon)=\min\{ t : D_{\tv}(t)\leq\epsilon\} \quad\mbox{ where }\quad D_{\tv}(t)=\max_{x}\|P^t(x,\cdot)-\pi\|_\tv\,.\]
A sequence of finite ergodic Markov chains is said to exhibit the \emph{cutoff phenomenon} if its total-variation distance from stationarity drops abruptly, over a period of time referred to as the \emph{cutoff window}, from near 1 to near 0; that is, there is cutoff iff $\tmix(\epsilon)=(1+o(1))\tmix(\epsilon')$ for any fixed $0<\epsilon,\epsilon'<1$. This phenomenon,  discovered by Aldous and Diaconis in the early 1980's (see~\cite{Aldous,AD,DiSh,Diaconis}), was verified only in relatively few cases, though believed to be widespread (cf.~\cite[\S18]{LPW09}).

Peres conjectured in 2004 that simple random walk (SRW) on every family of transitive expanders exhibits cutoff, yet there was not a single example of such a family
prior to the result on Ramanujan graphs in~\cite{Lubetzky2016}. The main step there was to show, via spectral analysis, that the \emph{nonbacktracking} random walk (NBRW) --- that does not traverse the same edge twice in a row --- has cutoff at time $\log_{k-1} n$ on any $k$-regular $n$-vertex Ramanujan graph. As it turns out, spectral analysis of the SRW does not directly establish $L^1$-cutoff, as the SRW has $L^1$ and $L^2$ cutoffs at different times. On the other hand, for the NBRW --- which may be viewed as the version of SRW that, on the universal cover of the graph (the $k$-regular tree), never creates cycles --- the locations of the $L^1$ and $L^2$ cutoffs do coincide. 

The analog of this phenomenon in higher dimensions would say that if the random walk operator $T$ satisfies a combinatorial property on the building $\cB$ (generalizing  the notion of avoiding cycles on the universal cover) then it should exhibit cutoff on any Ramanujan complex $X=\Gamma\backslash\cB$ as defined above. For the NBRW, the digraph $\cY_{T,\fC}$ associated with the corresponding operator $T$ on the universal cover is nothing more than the infinite directed (away from the root) $k$-regular tree. The combinatorial criterion defined next broadens the permissible digraphs $\cY_{T,\fC}$ supporting cutoff for $T$ on any Ramanujan complex $X$.  

 Call a $k$-branching operator $T$ \emph{collision-free} if its associated digraph $\cY_{T,\fC}$ has at most one (directed) path from $x$ to $y$ for every  $x,y\in\fC$.
Our main result  is that this combinatorial criterion implies cutoff for all Ramanujan complexes. \begin{maintheorem}
\label{thm:Ramanujan-cutoff}Let $\mathcal{B}$ be the Bruhat--Tits
building associated with a simple algebraic group $G$ of rank $d\geq 1$ over a locally compact
non-Archimedean field. Let $T$ be a $G$-equivariant and collision-free
$k$-regular branching  operator on a subset of cells $\fC \subseteq \mathcal{B}$. Let $X = \Gamma\backslash\cB$ be any Ramanujan complex as defined in~\eqref{eq:ram-def},
let $\fX=\Gamma\backslash\fC\subseteq X$, and set $n=|\fX|$. Let $\rho(x,y) = \min\{\ell : y\in T^\ell(x)\}$ for $x,y\in\fX$ (shortest-path metric in  $\cY_{T,\fX}$), and w.l.o.g.\ assume $\max_{x,y} \rho(x,y)<\infty$.
\begin{enumerate}[(i)]
\item
 There exist $c_G,M_G>0$ depending only on $G$, such that the following holds. If $Y_t$ is the random walk associated with $T$ on $\fX$, modified in its first step to perform $U_0\sim\mathrm{Uniform}(\{1,\ldots,M_G\})$ steps instead of one, then for every fixed $0<\delta<1$ and $n$ large enough,
 \begin{equation*}
 \left| \tmix(\delta) -  \log_k n \right | \leq c_G \log_k\log n\,.	
\end{equation*}
In particular, $Y_t$ exhibits worst-case total-variation cutoff at time $\log_k n$.

\item   There exists $c_G>0$, depending only on $G$, such that for every $x\in \fX$,
  \[ \#\left\{ y \in \fX : \left| \rho(x,y) - \log_{k} n \right| > c_G \log_{k}\log n \right\} =o(n)\,.\]
\end{enumerate}
\end{maintheorem}

Representation theory has been used to show cutoff for random walks on groups, notably in the pioneering work of Diaconis and Shahshahani~\cite{DiSh} (for more on this, see~\cite{Diaconis03}, the monograph by Diaconis~\cite{Diaconis88} and survey by Saloff-Coste~\cite{SaloffCoste04}); however, this is perhaps the first case where the infinite dimensional representation theory of locally compact groups is used to study such problems.

In~\S\ref{sub:branching-on-simple-gps}, we show that operators such as $T$ above, acting on the chambers of the building, are parameterized as $T = T_{w_0}$, where $w_0$ is an element of the affine Weyl group $\cW$ of $G$. In Theorem~\ref{thm:bruhat-collision-free}, we give a simple necessary and sufficient condition on $w_0\in \cW$ for $T_{w_0}$ to be collision-free.
Moreover, we show that various geodesic flow operators on cells of all dimensions (see Definition~\ref{def:geodesic-flow}) 
are collision-free.

\begin{example}
For $d=2$, one has $j$-dimensional geodesic flow operators ($j=1,2$) on tripartite Ramanujan complexes, with the color of a vertex $x$ denoted by $\col(x)$:
\[
\begin{tabular}{cccc}
\toprule%
$j$ & $k$-branching geodesic flow operator $T$ & $k$  & cutoff location  \\
\midrule
\noalign{\medskip}
$1$ 
&  
\begin{tabular}{c} 
$\fX = \left\{ (x,y)\in X: \col(y)\equiv\col(x)+1\right\}$ \\
\noalign{\medskip}
$T(x,y) = \{(y,z)\in \fX: \{x,y,z\}\notin X\}$ 
\end{tabular} 
\raisebox{-0.25in}{
\begin{tikzpicture}
      \draw[color=gray,fill=lightgray] (0,0)--(0.25,0.43)--(0.5,0)--(0,0);
      \draw[color=gray,fill=lightgray] (0.51,0)--(0.76,0.43)--(1.01,0)--(0,0);
      \draw[color=black,thick,->] (0,0)--(0.5,0);
      \draw[color=black,thick,->] (0.51,0)--(1.01,0);   
      \draw[color=gray,thin,dashed] (0,0.0) arc (180:360:0.48) ;
	  \node[color=red] at (0.51,-0.45) {$\times$};
	  \node[font=\tiny] at (0,-0.15) {$x$};
	  \node[font=\tiny] at (0.5,-0.2) {$y$};
	  \node[font=\tiny] at (1,-0.15) {$z$};
  \end{tikzpicture}}
&
  $q^2$ & $\frac12\log_q n$ \\
\midrule[0.25pt]
\noalign{\medskip}
$2$ &  \begin{tabular}{c} 
$\fX = \left\{ (x,y,z)\in X: \begin{array}{l}\col(y)\equiv\col(x)+1\\ 
 \col(z)\equiv\col(y)+1\end{array}
\right\}$ \\
\noalign{\medskip}
$T(x,y,z) = \{(y,z,w) \in \fX\,,\,w\neq x\}$ 
\end{tabular}
\hspace{-0.2in}
\raisebox{-0.2in}{
\begin{tikzpicture}
      \draw[color=black,thick,fill=lightgray] (0,0)--(0.25,0.43)--(0.5,0)--(0,0);
      \draw[color=black,thick,fill=lightgray] (0.3,0.43)--(0.8,0.43)--(0.55,0)--(0.3,0.43);
	  \node[font=\tiny] at (0,-0.15) {$x$};
	  \node[font=\tiny] at (0.52,-0.2) {$y$};
	  \node[font=\tiny] at (0.26,0.57) {$z$};
	  \node[font=\tiny] at (0.75,0.57) {$w$};

  \end{tikzpicture}}
 & $q$ & $\log_q n$ \\
\noalign{\medskip}
\midrule[0.25pt]
\bottomrule
\end{tabular}
\]
See~\S\ref{sub:Geodesic-flow} for the general case, where the $j$-dimensional geodesic flow ($1\leq j \leq d$) has cutoff at time $(d+1-j)^{-1} \log_q n$.
\label{ex:geod-flow}
\end{example}
 
The  cutoff locations and shortest-path lengths in Theorem~\ref{thm:Ramanujan-cutoff} are optimal among any operator with maximal out-degree $k$: the walk cannot reach a linear number of vertices, let alone mix, by time $(1-\epsilon)\log_k n$. This shows the extremality of such walks on Ramanujan complexes vs.\ any $k$-out-regular walk on an $n$-element~set.

Let us stress another interesting point: in recent years, Ramanujan complexes have been instrumental in tackling several open problems (see, e.g.,~\cite{Evra2014mixing,Evra2015systolic,FGL+11,Kaufman2016isoperimetric}). In all these applications, one could settle for less than the Ramanujan property, and in fact, a quantitative version of Kazhdan's Property (T) would suffice (see for instance~\cite{Evra2015finite}).
In this work, on the other hand, the Ramanujan property is crucial for the analysis of the walk operators on the complexes to be sharp. 

An important tool in this work is the directed analog of Ramanujan graphs. We define a finite $k$-out-regular digraph $\cY$ to be Ramanujan if the following holds:

\vspace{-0.35cm}
\begin{equation}\begin{tabular}{l}every eigenvalue $\lambda\in\C$ of the adjacency matrix   \\
of $\cY$ satisfies either $\left|\lambda\right|\leq\sqrt{k}$
or $\left|\lambda\right|=k$\end{tabular}\,.\label{eq:ram-digrap-def}
\end{equation}
\vspace{-0.3cm}

In these terms, the fundamental result of Hashimoto~\cite{hashimoto1989zeta} which characterized the spectrum of the NBRW operator on a graph, shows
that if $\cG$ is a $k$-regular undirected Ramanujan graph, then the digraph $\cY_{T,\cG}$ corresponding to the NBRW operator $T$ on $\cG$ is a Ramanujan digraph. This spectral property was one of two key ingredients in the mixing time analysis of the NBRW in~\cite{Lubetzky2016}; the other --- an algebraic property (established in that work for any graph) --- was that the adjacency matrix of $\cY_{T,\cG}$ is unitarily similar to a block-diagonal matrix with blocks of size at most $2$ and bounded entries.

Generalizing this notion, we call a matrix \emph{$r$-normal} if it is unitarily similar to a block-diagonal matrix where each of the blocks has size at most $r$, and say that an operator $T$ on $\fX\subset X$  is $r$-normal if the adjacency matrix of $\cY_{T,\fX}$ is $r$-normal.
 We further call $T$ \emph{irreducible} if $\cY_{T,\fX}$ is strongly-connected, and then say its \emph{period} is the $\gcd$ of all directed cycles in $\cY_{T,\fX}$.

To prove the main theorem, we show that under its hypotheses,  $T$ is $r$-normal for some $r=r(G)$; furthermore, its corresponding digraph $\cY_{T,\fX}$ is Ramanujan. Proposition~\ref{prop:cutoff-digraph}, which is of independent interest, establishes total-variation cutoff for such digraphs by following the argument of~\cite[Theorem~3.5]{Lubetzky2016}, extending to general $r$ what was proved there for $r=2$. As shown in that work for graphs (see~\cite[Corollaries~2 and~3]{Lubetzky2016}; also see~\cite[Theorem~1.3]{Sardari} for an alternative proof, discovered independently, of the result on typical distances  in Ramanujan graphs), this allows one to estimate key geometric features of Ramanujan digraphs:  

\vspace{-0.5cm}
\begin{maintheorem}\label{thm:distances}
For fixed  $k, M,r\geq 1$, let $\cY$ denote a $k$-regular strongly-connected Ramanujan digraph on $n$ vertices, and suppose that its adjacency matrix $A_\cY$ has period $M$ and is $r$-normal.
For every $x\in V(\cY)$,
there is a directed
path of length at most $\log_{k}n + (2r-1+o(1))\log_k\log n$ from $x$ to $y$, and such a path from $y$ to $x$, for all but $o(n)$ of the vertices $y\in V(\cY)$.
In particular, $\diam(\cY)\leq(2+o(1))\log_{k}n$.
\end{maintheorem}

Theorem~\ref{thm:Ramanujan-cutoff} is deduced from Proposition~\ref{prop:cutoff-digraph} and Theorem~\ref{thm:distances} using the following result, which reduces branching operators on Ramanujan complexes to  digraphs.

\begin{maintheorem}\label{thm:T-collision-free-r-normal-Ramanujan}Let $T$ denote
a collision-free, $G$-equivariant, $k$-regular branching operator
on $\fC\subseteq\mathcal{B}$, let $X=\Gamma\backslash\mathcal{B}$
be a Ramanujan complex as defined in~\eqref{eq:ram-def}, and let $\fX=\Gamma\backslash\fC$. Then the
digraph $\cY_{T,\fX}$ is Ramanujan and $r$-normal for  $r=r(G)$.\end{maintheorem}

\begin{remark}
We stress that $r$ cannot be bounded as a function of $k$. Such is the case for the geodesic flow operators mentioned in Example~\ref{ex:geod-flow} in general dimension (see Propositions~\ref{prop:j-flow-(m)j-normal} and~\ref{prop:flow-not-m-1-normal}).	
\end{remark}
\vspace{-0.25cm}
\begin{remark}
Passing from  Ramanujan complexes to  Ramanujan digraphs is by no means a tautology; e.g., neither the graphs nor digraphs associated with SRW on the 1-skeleton/chambers of a Ramanujan complex are Ramanujan (cf.~\cite{Lubotzky2005a}). 	
\end{remark}
\vspace{-0.05cm}

The reduction step from Ramanujan complexes to Ramanujan digraphs
uses  representation theory and harmonic analysis of $G$:
If $\lambda$ is an eigenvalue of the digraph
induced by $T$ on $\fX=\Gamma\backslash\fC$, with an eigenfunction
$f$, then $f$ induces a matrix coefficient $\varphi_{f}$ on $G$,
which is a $\lambda$-eigenfunction of $T$ on $\fC$. If
$X=\Gamma\backslash\mathcal{B}$ is Ramanujan, then $\varphi_{f}$
is a matrix coefficient of a tempered representation. A general result
on semisimple groups in~\cite{Haagerup1988} ensures that $\varphi_{f}$
is in $L^{2+\varepsilon}$ for all $\varepsilon>0$. We then show
that an $L^{2+\varepsilon}$ eigenvalue $\lambda$ of a collision-free
$k$-regular operator $T$ satisfies $\left|\lambda\right|\leq\sqrt{k}$,
which implies that $\cY_{T,\fX}$ is a Ramanujan digraph.
Using  that the space $L^{2}\left(\Gamma\backslash G/P\right)$,
when $P$ is a parahoric subgroup, is decomposed as a direct sum of
irreducible representations of the Hecke--Iwahori algebra, whose irreducible
representations are all of bounded dimension, we deduce that it is
also $r$-normal for some $r$ (see~\S\ref{sec:Ramanujan-complexes}).

Finally, our analysis of random walk operators also implies a result
on the zeta functions of Ramanujan complexes. Recall (see~\S\ref{sec:Zeta-functions}
for further details) that Ihara~\cite{ihara1966discrete} and Hashimoto~\cite{hashimoto1989zeta}
associated zeta functions with $k$-regular (di)graphs, and showed
that the graph is Ramanujan iff the zeta function satisfies
the Riemann hypothesis (R.H.). The zeta function of a digraph $\cY$
is  
$Z_{\cY}\left(u\right)=\prod_{\left[\gamma\right]}\frac{1}{1-u^{\ell\left(\gamma\right)}}
$,
where $\left[\gamma\right]$ runs over equivalence classes of
primitive directed cycles of $\cY$ (see~\S\ref{sec:Zeta-functions}).
We say that it satisfies the R.H.\ if $Z_{\cY}\left(k^{-s}\right)=\infty$
implies $|s|=1$ or $\Re\left(s\right)\leq\frac{1}{2}$.

Following W.\ Li, various authors have defined zeta functions associated with Ramanujan
complexes, and verified the R.H.\ for them over groups of type $\widetilde A_n$ and~$\widetilde C_2$ (cf.~\cite{kang2010zeta,fang2013zeta,kang2014zeta,kang2015zeta,kang2016riemann}).  These can all be described in our notation as $Z_{\cY_{T,\fX}}\left(u\right)$  for suitable collision-free
branching operators $T$  on $\mathcal{B}$ (see~\S\ref{sec:Zeta-functions}). 
As a corollary of Theorem~\ref{thm:T-collision-free-r-normal-Ramanujan}, we thus have a far reaching generalization of all of these results:
\begin{maincoro}
\label{cor:R-H-for-Ramanujan} Let $X$ be a Ramanujan complex as in~\eqref{eq:ram-def}, and let $Z_{\cY_{T,\fX}}(u)$ be
the zeta function associated with a $G$-equivariant, collision-free, $k$-regular branching operator $T$
on $\cB=\mathcal{B}\left(G\right)$. Then $Z_{\cY_{T,\fX}}(u)$ satisfies the Riemann Hypothesis;
that is, if $Z_{\cY_{T,\fX}}\left(u\right)$ has a pole at $k^{-s}$
then $|s|=1$ or $\Re\left(s\right)\leq\frac{1}{2}$.
\end{maincoro}
We should mention that Kang~\cite{kang2016riemann} (and the other
references mentioned above) obtained more detailed
information about the location of the poles, by carefully examining the representations appearing in these cases.

Note that Corollary~\ref{cor:R-H-for-Ramanujan} is only part of
the story; as asked explicitly by Kang~\cite{kang2016riemann}, one would like to have the
converse, i.e., that the R.H.\ implies the Ramanujan property. This (and much more)
has been established recently by Kamber~\cite{Kamber2017LP}. The reader is referred
to his paper for a wealth of interesting related results.

\section{Spectrum of branching operators}\label{sec:Spectrum}

Let $X$ be a countable set and, as usual, let $L^{p}=L^{p}\left(X\right)$ for $1\leq p\leq\infty$ denote its complex $L^{p}$-space. Every $k$-branching
operator $T$ on $X$ gives rise to a bounded
linear map $A=A_{\cY_{T,X}}:L^{p}\left(X\right)\rightarrow L^{p}\left(X\right)$
for every $p$, where $A_{\cY_{T,X}}$ corresponds to  the adjacency matrix of $\cY_{T,X}$, the digraph associated with $T$; namely,
\[
\left(A f\right)\left(x\right)=\sum_{y\in T\left(x\right)}f\left(y\right)\,.
\]
We are mostly interested in the
spectrum of $A$ on $L^{2}$, but our analysis will also consider it
as acting on $L^{p}$ for $p\neq 2$.
\begin{definition} Let $A$ be the map associated with a $k$-branching operator~$T$.
\begin{enumerate}
\item The \emph{spectrum} of $A$ (on $L^{2}$) is
\[
\mathrm{spec}\left(A\right)=\Bigl\{ \lambda\in\mathbb{C}\,:\;\left(\lambda I-A\right)\big|_{L^{2}}\text{ does not have a bounded inverse}\Bigr\} ,
\]
and the \emph{spectral radius }of $A$ is $\rho\left(A\right)=\max\left\{ \left|\lambda\right|\,:\;\lambda\in\mathrm{spec}\left(A\right)\right\}$.
\item The $p$-\emph{point spectrum} of $A$ is 
\[
\mathrm{Pspec}_{p}\left(A\right)=\Bigl\{ \lambda\in\mathbb{C}\,:\;\exists f\in L^{p}\text{ such that }Af=\lambda f\Bigr\} .
\]
\item The \emph{$2^+$-point spectrum} of $A$ is 
\[
\mathrm{Pspec}_{2^+}\left(A\right)=\biggl\{ \lambda\in\mathbb{C}\,:\;\exists f\in\bigcap_{\varepsilon>0}L^{2+\varepsilon}\text{ such that }Af=\lambda f\biggr\} .
\]
\item The \emph{approximate point spectrum} of $A$ is 
\[
\mathrm{APspec}\left(A\right)=\biggl\{ \lambda\in\mathbb{C}\,:\; \exists f_{n}\in L^{2}\text{ such that }\frac{\left\Vert Af_{n}-\lambda f_{n}\right\Vert }{\left\Vert f_{n}\right\Vert }\rightarrow0\biggr\} .
\]
\end{enumerate}
\end{definition}
For finite $X$, all these definitions are equivalent,
but otherwise they may differ. In particular, while $\mathrm{APspec\left(A\right)}=\mathrm{spec}\left(A\right)$
holds for a \emph{normal} operator $A$, this may not hold in general; in this work we will (mainly) consider operators that are not normal, so this distinction
is significant for us. We start with a trivial bound:
\begin{proposition}
\label{prop:bound-deg}If $T:X\rightarrow2^{X}$ is a branching operator
with in-degrees and out-degrees bounded by $k$, then 
$
\rho\left(A_{\cY_T}\right)\leq\left\Vert A_{\cY_T}\right\Vert _{2}\leq k\,.
$
\end{proposition}
\begin{proof}
The bound $\rho\left(A_{\cY_T}\right)\leq\left\Vert A_{\cY_T}\right\Vert _{2}$
holds for any bounded operator, and if $f\in L^{2}\left(X\right)$
with $\left\Vert f\right\Vert _{2}=1$ then 
\[
\left\Vert A_{\cY_T}f\right\Vert _{2}^{2}=\sum_{x\in X}\Bigl| \sum_{y\in T\left(x\right)}f\left(y\right)\Bigr|^{2}\leq k\sum_{x\in X}\sum_{y\in T\left(x\right)}\left|f\left(y\right)\right|^{2}\leq k^{2}\sum_{y\in X}\left|f\left(y\right)\right|^{2}=k^{2}\,,
\]
where the first inequality is by Cauchy--Schwarz.
\end{proof}
Clearly, if $T$ is a $k$-branching collision-free operator, then the associated digraph $\cY_T$  must be infinite. Such
a digraph looks like a tree from the point of view of each $y_{0}\in \cY$.
Still we should warn that, once the orientation is ignored, it may
have non-trivial closed paths. Moreover, its spectrum can be very
different from that of a tree --- see Example \ref{exa:amenable-digraph}
for an ``amenable'' collision-free digraph. Still, the $2^+$-point spectrum of such a graph is very limited. 
\begin{proposition}
\label{prop:collision-free-riemann} Let $\cY=\cY_T$ for a $k$-branching collision-free operator $T$. Then every $\lambda\in \Pspec_{2^+}(A_{\cY})$ satisfies $\left|\lambda\right|\leq\sqrt{k}$.
\end{proposition}
\begin{proof}
Let $f\in L^{\infty}\left(\cY\right)$ be an eigenfunction of $A_{\cY}$
with eigenvalue $\lambda$. Scale $f$ so that $f\left(v\right)=1$
for some $v$, and fix $\varepsilon>0$. Let us show that $T^{j}\left(v\right)$, the $j$-th level in
the subtree of $\cY$ emanating from $v$, satisfies
\[
\sum_{w\in T^{j}\left(v\right)}\left|f\left(w\right)\right|^{2+\varepsilon}\geq\biggl|\frac{\lambda^{2+\varepsilon}}{k^{1+\varepsilon}}\biggr|^{j}.
\]
This indeed holds for $j=0$. Notice that if $w,w'$ are different
vertices in $T^{j}\left(v\right)$, then $T\left(w\right)$ and $T\left(w'\right)$
are disjoint, since $\cY$ is collision-free. This allows us to proceed
by induction:
\begin{align*}
\sum_{w\in T^{j+1}\left(v\right)}&\left|f\left(w\right)\right|^{2+\varepsilon}  =\sum_{u\in T^{j}\left(v\right)}\sum_{w\in T\left(u\right)}\left|f\left(w\right)\right|^{2+\varepsilon}
\geq \frac{1}{k^{1+\varepsilon}}\sum_{u\in T^{j}\left(v\right)}\biggl(\sum_{w\in T\left(u\right)}\left|f\left(w\right)\right|\biggr)^{2+\varepsilon}\,,
\end{align*}
where the last inequality is a consequence of H\"older's inequality, since 
 for every H\"older conjugates $p,q$ one has $\sum_{i=1}^{k}\left|a_{i}\right|^{p}\geq k^{-p/q}(\sum_{i=1}^{k}|a_{i}|)^{p}$. The right-hand of the last display is, in turn, at least
\begin{align*}
\frac{1}{k^{1+\varepsilon}}\sum_{u\in T^{j}\left(v\right)}\left|\lambda f\left(u\right)\right|^{2+\varepsilon}=\frac{\left|\lambda\right|^{2+\varepsilon}}{k^{1+\varepsilon}}\sum_{u\in T^{j}\left(v\right)}\left|f\left(u\right)\right|^{2+\varepsilon}\geq\left|\frac{\lambda^{2+\varepsilon}}{k^{1+\varepsilon}}\right|^{j+1}\,.
\end{align*}
Since $\cY$ is collision-free,
the levels $T^{j}\left(v\right)$ are disjoint for different $j$'s,
so that 
\[
\left\Vert f\right\Vert _{2+\varepsilon}^{2+\varepsilon}\geq\sum_{j=0}^{\infty}\sum_{w\in T^{j}\left(v\right)}\left|f\left(w\right)\right|^{2+\varepsilon}\geq\sum_{j=0}^{\infty}\left|\frac{\lambda^{2+\varepsilon}}{k^{1+\varepsilon}}\right|^{j},
\]
which is finite iff $\left|\lambda\right|<k^{\frac{1+\varepsilon}{2+\varepsilon}}$.
Thus if, $f\in\bigcap_{\varepsilon>0}L^{2+\varepsilon}\left(\cY\right)$,
then $\left|\lambda\right|\leq\sqrt{k}$.
\end{proof}
However, it is not necessarily the case that an arbitrary $\lambda$
in the spectrum of $A_{\cY}$ on $L^{2}\left(\cY\right)$ has $|\lambda|\leq\sqrt{k}$, and so $\cY$ is
not necessarily Ramanujan; indeed, the following example shows that the situation for the full spectrum of a collision-free
digraph can be drastically different from the $2^+$-point spectrum.
\begin{example}[Amenable collision-free digraph]
\label{exa:amenable-digraph}Let $V=\mathbb{N}\times\left(\nicefrac{\mathbb{Z}}{k\mathbb{Z}}\left[x\right]\right)$
with each $\left(n,p\left(x\right)\right)$ connected into the $k$
vertices $\left(n+1,p\left(x\right)+a\cdot x^{n}\right)$ with $a\in\nicefrac{\mathbb{Z}}{k\mathbb{Z}}$.
This digraph is $k$-out-regular, and for any $m\in\mathbb{N}$, one can take
\[
f=\one_{\left\{ \left(0,p\left(x\right)\right)\,:\;\deg p(x)<m\right\} }\,,
\]
for which $T^{m}f=k^{m}\cdot\one_{\left\{ \left(m,p\left(x\right)\right)\,:\;\deg p(x)<m\right\} }$\,,
implying 
\[
\left\Vert T^{m}\right\Vert \geq\frac{\left\Vert T^{m}f\right\Vert }{\left\Vert f\right\Vert }=\frac{k^{3m/2}}{k^{m/2}}=k^{m}\,.
\]
This gives $\rho\left(T\right)=k$, since $\rho\left(T\right)\leq k$
by Proposition \ref{prop:bound-deg} and in general one has $\rho\left(T\right)=\lim_{m\rightarrow\infty}\sqrt[m]{\left\Vert T^{m}\right\Vert }$.
In fact, $k$ is even in the approximate point spectrum of $T$, since, for
$f_{m}=\one_{\left\{ \left(j,p\left(x\right)\right)\,:\;j<m,\,\deg p(x)<m\right\} }$,
\[ T(f_m) = k \cdot \one_{\left\{ (j,p(x)) \,:\; 1\leq j \leq m\,,\, \deg p(x)<m\right\}}\,,
\]
and hence $\left\Vert \left(T-k\right)f_{m}\right\Vert /\left\Vert f_{m}\right\Vert \overset{{\scriptscriptstyle m\rightarrow\infty}}{\longrightarrow}0$.
\end{example}
In~\S\ref{sec:Ramanujan-complexes}
we will show that in the cases we focus on --- affine buildings associated with simple groups --- spectral behaviors as in the example above cannot
occur.

\section{Random walks on Ramanujan digraphs}\label{sec:Random-walks-digraphs}
The goal of this section is to establish the following cutoff result for Ramanujan digraphs, generalizing~\cite[Theorem~3.5]{Lubetzky2016} from $r=2$ to any fixed $r$.

In what follows, say that a matrix $A$ indexed by $V\times V$ is $r$-normal w.r.t.\ a distribution $\pi$ on $V$ if there exists a block-diagonal matrix $\Lambda$, whose blocks are each of size at most $r$, and a basis $\{w_i\}_{i\in V}$ of $\C^V$, such that $A W = W \Lambda$ for the matrix  $W$ which has the $w_i$'s as its columns, and the $w_i$'s form an orthonormal system w.r.t.\ the inner product $\left<f,g\right>_{L^2(\pi)} := \sum_{x\in V} f(x) \overline{g(x)}\pi(x)$.
We will look at  distributions $\pi$ where the ratio between the maximum and minimum probabilities is at most $\log^\alpha |V|$ for some $\alpha\geq 0$ (in our main result $\pi$ is uniform, i.e., $\alpha=0$).

\begin{proposition}
\label{prop:cutoff-digraph}
Fix  $k, b,M,r\geq 1$, $\alpha\geq0$, and let $\cY$ be a $k$-out-regular $n$-vertex  strongly-connected digraph, with  all in-degrees at most $b$, and a period dividing $M$.
Let $Y_t$ be \emph{SRW}  on $\cY$, modified  to make $U_0\sim\mathrm{Uniform}(\{1,\ldots,M\})$ steps in its first move, let $\pi$ be its stationary distribution, and suppose $\|\pi\|_{\infty} / \|\pi^{-1}\|_\infty \leq \log^\alpha n$.
If $\cY$ is Ramanujan and its adjacency matrix $A_\cY$ is $r$-normal w.r.t.\ $\pi$, then 
\begin{equation*}
 \left| \tmix(\delta) -  \log_k n \right | \leq (2r-1+\alpha+o(1))\log_k\log n\qquad\mbox{ for every fixed $0<\delta<1$}\,.
\end{equation*}
In particular, $Y_t$ exhibits worst-case total-variation cutoff at time $\log_k n$.
\end{proposition}
\begin{proof}
The lower bound on $\tmix(\delta)$ will follow immediately from the following elementary lower bound on the mixing time of random walk (cf.~\cite[Claim~4.8]{LS-gnd}).
\begin{observation} Let $(Y_t)$ be simple random walk on a digraph $\cY$ with $n$ vertices and maximum out-degree $k\geq 2$, and suppose that its stationary distribution $\pi$ satisfies $\|\pi\|_\infty \leq (\log n)^\alpha /n$ for some $\alpha>0$. Then
\begin{equation}
  \label{eq-dtv-lower-bound}
  D_{\tv}(T) \geq 1-1/\log n \quad \mbox{ at }\quad T =\lfloor \log_{k} n  - (\alpha+1)\log_{k}\log n\rceil\,.
\end{equation}
\end{observation}
Indeed, for any initial vertex $x_0\in V(\cY)$, if $S$ is the set of vertices supporting the distribution of $Y_T$ then $|S| \leq k^T \leq n/(\log n)^{\alpha+1}$ gives $\pi(S) \leq \|\pi\|_{\infty} |S| \leq 1/\log n$.

For a matching upper bound, first consider the aperiodic case $M=1$.
Since $\cY$ is strongly-connected, this corresponds to the case where the only eigenvalue $\lambda$ with $|\lambda|=k$ is $k$ itself, which appears with multiplicity 1.
We will show that if $\nu_{t}^{x_0} = \P_{x_0}(Y_t \in \cdot)$ is the distribution of $(Y_t)$ starting from $X_0=x_0$, then 
 \begin{equation}\label{eq-dtv-tsar}
  \max_{x_0\in V(\cY)} \left\| \nu_{t_\star}^{x_0}/ \pi - \underline{1} \right\|^2_{L^2(\pi)} \leq \left(\frac{r^3 b^{r-1}}{(\log k)^{2r-2}} + o(1)\right) \frac1{\log n}
 \end{equation}
 at 
\begin{equation}\label{eq-tstar}
  t_\star = \left\lceil \log_{k} n + (2r-1+\alpha)\log_{k}\log n\right\rceil\,. 
  \end{equation}
In particular, this will imply that $\max_{x_0}\|\nu_{t_\star}^{x_0}/\pi-1\|_{L^1(\pi)} =
O(1/\sqrt{\log n})$, and so
 $\tmix(\delta)\leq t_\star$ for every fixed $\delta>0$ and large enough $n$, concluding the case $M=1$.
  
Let $A_\cY W = W\Lambda$ be the matrices per the $r$-normal decomposition of $A$. That is,
$\Lambda=\diag(B_1,\ldots,B_m)$, where $B_l$ $(l=1,\ldots,m)$ is an $s_l\times s_l$ upper triangular matrix with $s_l \leq r$, and $W$ is a matrix whose columns, denoted $(w_{l,i})_{1\leq l \leq m,\, 1\leq i \leq s_l}$,  form an orthonormal system w.r.t.\ $\left<\cdot,\cdot\right>_{L^2(\pi)}$. 

Since $k^{-1} A_\cY$ has $\pi$ as a left eigenvector with eigenvalue $1$ (by  definition of $\pi$),  it follows that every right eigenvector $v$ of $A_\cY$ with eigenvalue $\lambda\neq k$ is orthogonal to $\pi$ w.r.t.\ the standard complex inner product. Equivalently, every such vector $v$ satisfies $\left<v,\underline{1}\right>_{L^2(\pi)} = 0$, and in particular, since the eigenvalue $\lambda=k$ is simple, we can assume w.l.o.g.\ that $w_{1,1} = \underline{1} $, corresponding to $s_1=1$ and $B_1 = \left(k\right)$ in $\Lambda$, while each $B_l$ for $l>1$ has diagonal entries $(\lambda_{l,1},\ldots,\lambda_{l,s_l})$ satisfying $|\lambda_{l,i}|\leq \sqrt{k}$ for all $1\leq i\leq s_l$.

We claim that, with this notation,
\begin{equation}\label{eq-Bl-entry-bound}
\max_{i,j} |B_l(i,j)| \leq \sqrt{kb} \quad\mbox{ and }\quad  \|B_l\|_F \leq \sqrt{r k b}\,,
\end{equation}
where $\|\cdot\|_F$ is the Frobenius norm.
Indeed, 
\[\left\Vert A_{\cY}\right\Vert _{2}^{2}=\rho\left(A_{\cY}A_{\cY}^{*}\right)\leq\left\Vert A_{\cY}A_{\cY}^{*}\right\Vert _{\infty}\leq\left\Vert A_{\cY}\right\Vert _{\infty}\left\Vert A_{\cY}^{*}\right\Vert _{\infty}=\left\Vert A_{\cY}\right\Vert _{\infty}\left\Vert A_{\cY}\right\Vert _{1}=kb\,. \]
In particular, for every $1\leq j \leq s_l$, the  left-hand of  the identity
\[ A_\cY \, w_{l,j} = \sum_{i \leq j} B_l( i, j) w_{l,i} \]
has $\|A_\cY w_{l,j} \|_2^2 \leq \|A_\cY\|_{ 2}^2 \leq k b$, and so applying Parseval's identity to its right-hand yields
\[ kb \geq \max_{1\leq j \leq s_l} \Big\|\sum_{i\leq j} B_l(i,j) w_{l,i} \Big\|_2^2 = \max_{1\leq j \leq s_l} \sum_{i\leq j} \left| B_l(i,j)\right|^2\,,\]
thus implying~\eqref{eq-Bl-entry-bound}.

For $l\geq 2$, every  $B_l$ may be written as $D_l + N_l$, where $D_l = \diag(\lambda_{l,1},\ldots,\lambda_{l,s_l})$ and $N_l$ is a nilpotent matrix of index at most $r$. We can therefore infer from~\eqref{eq-Bl-entry-bound} and the bound $\max_j |\lambda_{l,j}|\leq\sqrt{k}$ that, for all $l\geq 2$ and every $t$,
\begin{equation}\label{eq-Bl-power-bound} \max_{i,j} \left|B_l^t(i,j)\right| 
 \leq \sum_{h=0}^{r-1} \binom{t}{h}  k^{(t-h)/2} (kb)^{h/2} \leq r \big(t\sqrt{b}\big)^{r-1} k^{t/2}\,.
 \end{equation} 

This bound will allow us to bound the $L^2$-distance of $\nu_t^{x_0}$ from equilibrium. 
Recall that $\nu_t^x(y) = k^{-t} A_\cY^t (x,y)$;
as $A_\cY W = W\Lambda$, it follows that for every $l,j,t$,
\[ A_\cY^t w_{l,j} = W \Lambda^t W^{-1} w_{l,j} =
\sum_{i=1}^{s_l}  B_l^t(i,j) w_{l,i}\,.\]
Combined with the decomposition $\delta_y = \sum_{l,j} \overline{w}_{l,j}(y) \pi(y) w_{l,j}$, this implies that 
\[ \nu_t^{x_0}(y) = k^{-t} \sum_{l=1}^m \sum_{i=1}^{s_l} \sum_{j=i}^{s_l} w_{l,i}(x_0) B_l^t(i,j) {\overline w}_{l,j}(y) \pi(y)\,,\]
and plugging in the fact that for $l=1$ we have $w_{1,1} = \underline{1}$ and $B_1=\left(k\right)$ then yields
\begin{align*}
  \nu_t^{x_0} /\pi   &= \underline{1}  + k^{-t} \sum_{l=2}^m \sum_{i=1}^{s_l} \sum_{j=i}^{s_l} w_{l,i}(x_0) B_l^t(i,j) {\overline w_{l,j}}\,.
\end{align*}
Since the $w_{l,i}$'s form an orthonormal basis w.r.t.\ $\left<\cdot,\cdot\right>_{L^2(\pi)}$, we get that
\begin{align*}
\bigl\|\nu_t^{x_0}/\pi &- \underline{1}\bigr\|^2_{L^2(\pi)} =  k^{-2t} \sum_{l=2}^m \sum_{j=1}^{s_l} \Bigl| \sum_{i=1}^{j}  w_{l,i}
(x_0)B_l^t(i,j)\Bigr|^2 \\
&\leq k^{-2t} \sum_{l=2}^m \sum_{j=1}^{s_l} j \sum_{i=1}^{j}  |w_{l,i}(x_0)|^2 \left|B_l^t(i,j)\right|^2 
\leq  r^3 b^{r-1}   t^{2(r-1)} k^{-t} \sum_{l,i} |w_{l,i}(x_0)|^2 \,,
\end{align*}
where the first inequality in the second line is by Cauchy--Schwarz, and the last one used~\eqref{eq-Bl-power-bound} and that $s_l\leq r$ for all $l$.
Finally, by Parseval's identity,
$ \sum_{l,i} |w_{l,i}(x_0)|^2$ is equal to $ \pi(x_0)^{-2} \|\delta_{x_0}\|_{L^2(\pi)}^2 = \pi(x_0)^{-1}$,
and therefore we conclude that 
 \begin{align}
\left\|\nu_t^{x_0}/\pi - \underline{1}\right\|^2_{L^2(\pi)} & \leq r^3 b^{r-1} t^{2(r-1)} \pi(x_0)^{-1} k^{-t}\,.
\label{eq-dtv-upper-bound}
\end{align}
Plugging in $t=t_\star$ from~\eqref{eq-tstar} in~\eqref{eq-dtv-upper-bound} establishes~\eqref{eq-dtv-tsar}, as desired.

It remains to treat the case $M\geq2$, allowing periodicity. Denote by $m$  the period of $A_\cY$ (so that $m \mid M$ by assumption), let $V(\cY) =V_0 \cup \ldots\cup V_{m-1}$ be an $m$-partition of $\cY$ such that all the out-neighbors of a vertex $x\in V_j$ belong to $V_{j+1\!\!\pmod m}$, and let $\pi_j$ denote the uniform distribution over $V_j$.

We claim that, analogous to the bound~\eqref{eq-dtv-tsar}, we have for every $j=0,\ldots,M-1$,
 \begin{equation}\label{eq-dtv-tsar-period-M}
  \max_{x_0\in V_j} \left\| \nu_{t_{\star}+M}^{x_0}/ \pi_j - \underline{1} \right\|^2_{L^2(\pi_j)} \leq \left(\frac{r^3 b^{r-1}}{(\log k)^{2r-2}} + o(1)\right) \frac1{\log n}\,.
 \end{equation}
 To see this, recall that by the Perron--Frobenius Theorem for irreducible matrices and our hypothesis on $\cY$, all eigenvalues of $A_\cY$ either lie within the ball of radius $\sqrt{k}$, or they are the of the form $\exp(2\pi i j / m) k$ for $j=0,\ldots,m-1$ (each appearing with multiplicity $1$). In particular, as $m\mid M$, the digraph $\cY'_j $ corresponding to the $M$-th power of $\cY$ restricted to $V_j$ for a given $j$ (that is, $A_{\cY'_j} = A_{\cY}^M\restriction_{V_j}$) has a single multiplicity for the Perron eigenvalue $k' = k^M$ (the out-degree) and all of its other eigenvalues lie within the ball of radius $k^{M/2} = \sqrt{k'}$.
 	Therefore, the inequality~\eqref{eq-dtv-tsar} is valid for the digraph $\cY'$ at $t'_\star$ given by
 	\[ t'_\star = \lceil \log_{k'} n + (2r-1+\alpha)\log_{k'}\log n\rceil = \bigl\lceil M^{-1} \log_k n + M^{-1}(2r-1+\alpha)\log_k \log n\bigr\rceil.
 	\]
 	The proof of~\eqref{eq-dtv-tsar-period-M} now follows from the fact that $t'_\star$ steps in $\cY'$ correspond to $M t'_\star \leq t_\star + M $ steps of the random walk in the original digraph $\cY$, and the $L^2$-distance is monotone non-increasing. Consequently, recalling that  $(Y_t)$ is the random walk that is modified in its first step to walk according to $A_\cY^j$ where $j$ is uniformly chosen over  $1,\ldots,M$, its distance from the uniform distribution at time $t_\star+M= \log_k n + (2r-1+\alpha+o(1))\log_k \log n$ is $o(1$), as claimed.
\end{proof}

\begin{proof}[\textbf{\emph{Proof of Theorem~\ref{thm:distances}}}]
Fix $\epsilon > 0$. By Proposition~\ref{prop:cutoff-digraph}, the simple random walk $Y_t$ on the digraph $\cY$, modified in its first steps to perform $U_0\sim \mathrm{Uniform}(\{1,\ldots,M\})$ steps instead of a single one, satisfies that $\max_{x_0} \left\|\P_{x_0}(Y_{t_\star}\in\cdot)-\pi\right\|_\tv = o(1)$ at $t_\star$ from~\eqref{eq-tstar} (see~\eqref{eq-dtv-tsar}). In particular, for every initial vertex $x_0$, there is a directed path of length at most $t_\star +M$ from $x_0$ to all but at most $o(n)$ vertices of $\cY$. Applying the same argument on the reverse digraph --- the one whose adjacency matrix is $A_{\cY}^{\textsc{t}} = A_{\cY}^*$, which is again $k$-regular, $r$-normal and Ramanujan --- shows that for every $x_0$, there is a directed path of length at most $t_\star +M$ to $x_0$ from all but $o(n)$ vertices of $\cY$. Altogether, the diameter of $\cY$ is at most $2(t_\star + M)$. 
\end{proof}

\section{\label{sec:Ramanujan-complexes}Random walks on Ramanujan complexes}

In this section we establish Theorem~\ref{thm:T-collision-free-r-normal-Ramanujan}, showing that  collision-free operators on affine buildings
induce $r$-normal Ramanujan digraphs on their Ramanujan quotients;
this will be achieved by Propositions~\ref{prop:Ramanujan-L2e},~\ref{prop:T-collision-free-Ramanujan}
and~\ref{prop:YT-collision-free-r-normal}. We will then combine these results with Proposition~\ref{prop:cutoff-digraph} to conclude
the proof of Theorem~\ref{thm:Ramanujan-cutoff}.

Recalling the definition of a Ramanujan complex in~\eqref{eq:ram-def},
let us mention that for the special case where $G$ is of type $\widetilde{A}_{n}$
the definition here is a priori stronger than the one given in~\cite{Lubotzky2005a,li2004ramanujan}.
There the authors used a maximal compact subgroup $K$ ---
a stabilizer of a vertex --- instead of $\mathcal{I}.$
As $K$ contains $\mathcal{I}$, the definition here may be stronger.
But as observed in~\cite{first2016ramanujan}, the explicit examples
constructed in~\cite{Lubotzky2005b} are Ramanujan also with this
stronger definition, since the work of Lafforgue~\cite{lafforgue2002chtoucas},
upon which~\cite{Lubotzky2005b} is based, is valid for $L$-spherical
representations for every compact open subgroup $L$. In fact, we
do not know any example of $X$ which is Ramanujan in one sense and
not in the other (and for $\widetilde{A}_{d}$, $d=1,2$ these notions
are indeed known to be equivalent --- see~\cite{kang2010zeta}).
Anyway, the stronger definition we use here seems to be the ``right''
one (see~\cite{first2016ramanujan} for a thorough discussion) and
certainly the one which makes sense for all simple groups $G$; Note
that for all $G$'s as above, the Iwahori subgroups are uniquely determined
up to conjugacy, but the maximal compact subgroups are not.

This stronger definition allows us to deduce bounds on the norms of
essentially all the combinatorial/geometric operators acting on $X$,
not only those defined in terms of its $1$-skeleton, as done in~\cite{Lubotzky2005a}.
Sometimes,~\cite{FGL+11,Evra2014mixing} are good examples, the information
on the $K$-spherical representations suffices even for the study
of high dimensional issues, as $X$ is a clique complex which is completely
determined by its $1$-skeleton. Still, in the current paper we make
essential use of operators acting on the higher dimensional cells,
and we do need the full power of~\eqref{eq:ram-def}. 
\begin{proposition} \label{prop:Ramanujan-L2e}If $T$ is a $G$-equivariant
branching operator on $\fC\subset\mathcal{B}$, and $X=\Gamma\backslash\mathcal{B}$
is Ramanujan, then every nontrivial eigenvalue of $T$ acting on $\fX=\Gamma\backslash\fC$
is in the $2^{+}$-point spectrum of $T$ acting on $\fC$.
\end{proposition} \begin{remark} The trivial eigenvalues of $T$
are the ones which arise from one-dimensional representations of $G$,
as explained in the proof. In the case that $G$ is simple (which
one can assume, without loss of generality), an eigenvalue is trivial
if and only if the corresponding eigenfunction, when lifted from $\Gamma\backslash\fC$
to $\fC$, is $G$-invariant. \end{remark} \begin{proof} Let $\Sigma$
be a collection of faces of the chamber $\sigma_{0}$ such that $\fC=\coprod_{\sigma\in\Sigma}G\sigma$,
and denote by $P_{\sigma}$ the $G$-stabilizer of each $\sigma\in\Sigma$.\footnote{The proof  is simpler in the special case where $T$ acts only on the top-dimensional cells, i.e., $\fC=\mathcal{B}\left(d\right)$
and $\Sigma=\left\{ \sigma_{0}\right\} $; the reader may want to
 first focus on that case of the proposition.} For a representation $V$ of $G$ denote 
\[
V^{\Sigma}=\bigoplus_{\sigma\in\Sigma}V^{P_{\sigma}}\leq\bigoplus_{\sigma\in\Sigma}V\,,
\]
which is a representation space of the algebra $M_{\Sigma\times\Sigma}\left(\mathbb{C}G\right)$.
Since $L^{2}\left(\Gamma\backslash G\sigma\right)\cong L^{2}\left(\Gamma\backslash G/P_{\sigma}\right)\cong L^{2}\left(\Gamma\backslash G\right)^{P_{\sigma}}$,
we can identify $L^{2}\left(\fX\right)$ with $L^{2}\left(\Gamma\backslash G\right)^{\Sigma}$,
and in particular for $\Gamma=1$ we obtain $L^{2}\left(\fC\right)\cong L^{2}\left(G\right)^{\Sigma}$.
Explicitly, every cell in $\fX=\Gamma\backslash\fC$ is of the form $\Gamma g\sigma$
for some $g\in G$ and $\sigma\in\Sigma$, and for $\left(f_{\sigma}\right)_{\sigma\in\Sigma}\in\bigoplus_{\sigma\in\Sigma}L^{2}\left(\Gamma\backslash G\right)^{P_{\sigma}}$
we have 
\begin{equation}
L^{2}\left(\fX\right)\cong L^{2}\left(\Gamma\backslash G\right)^{\Sigma}\qquad\text{by}\qquad f\left(\Gamma g\sigma\right)=f_{\sigma}\left(\Gamma g\right)\,.\label{eq:identification}
\end{equation}
Turning to the operator $T$, let $S^{\sigma,\tau}$ be subsets of
$G$, indexed by $\sigma,\tau\in\Sigma$, such that 
\[
T\left(\sigma\right)=\left\{ s\tau\,\middle|\,\tau\in\Sigma,s\in S^{\sigma,\tau}\right\} \,.
\]
Since $T$ is $G$-equivariant, this implies that $T\left(g\sigma\right)=\left\{ gs\tau\,\middle|\,\tau\in\Sigma,s\in S^{\sigma,\tau}\right\} $
for any $g\in G$ and $\sigma\in\Sigma$, and thus $A_{\cY_{T,\fX}}$
acts on $L^{2}\left(\fX\right)$ by 
\[
\left(A_{\cY_{T,\fX}}f\right)\left(\Gamma g\sigma\right)=\sum_{\tau\in\Sigma}\sum_{s\in S^{\sigma,\tau}}f\left(\Gamma gs\tau\right)\,.
\]
Thus, under the identification~\eqref{eq:identification}, $A_{\cY_{T,\fX}}$
acts as the element $\mathcal{A}\in M_{\Sigma\times\Sigma}\left(\mathbb{C}G\right)$
with $\left(\mathcal{A}\right)_{\sigma,\tau}=\sum_{s\in S^{\sigma,\tau}}s$
(in particular, $V^{\Sigma}$ is $\mathcal{A}$-stable). Since $\Gamma$
is cocompact, $L^{2}\left(\Gamma\backslash G\right)$ decomposes as
an orthogonal sum of unitary irreducible representations of $G$,
\[
L^{2}\left(\Gamma\backslash G\right)=\bigoplus_{i}V_{i}\oplus\bigoplus_{j}W_{j}\oplus\bigoplus_{k}U_{k}\,,
\]
where $V_{i}$ are the infinite-dimensional $\mathcal{I}$-spherical
representations which appear in $L^{2}\left(\Gamma\backslash G\right)$,
$W_{j}$ are the finite-dimensional ones, and $U_{k}$ are the ones
with no $\mathcal{I}$-fixed vectors. This implies 
\begin{equation}
L^{2}\left(\fX\right)\cong L^{2}\left(\Gamma\backslash G\right)^{\Sigma}=\bigoplus_{i}V_{i}^{\Sigma}\oplus\bigoplus_{j}W_{j}^{\Sigma}\,,\label{eq:L2(X)-decomposition}
\end{equation}
and there are only finitely many summands in this decomposition, as
$X$ is finite. Since $A_{\cY_{T,\fX}}$ coincides with an element in $M_{\Sigma\times\Sigma}\left(\mathbb{C}G\right)$,
it respects the decomposition~\eqref{eq:L2(X)-decomposition} and
thus the spectrum of $T$ on $\fX$ is 
\[
\mathrm{spec}\left(A_{\cY_{T,\fX}}\right)=\bigcup_{i}\mathrm{spec}\left(\mathcal{A}\big|_{V_{i}^{\Sigma}}\right)\bigcup_{j}\mathrm{spec}\left(\mathcal{A}\big|_{W_{j}^{\Sigma}}\right)\,.
\]
The trivial eigenvalues of $T$ on $\fX$ are by definition the ones
in $\bigcup_{j}\mathrm{spec}\bigl(\mathcal{A}\big|_{W_{j}^{\Sigma}}\bigr)$.
When $G$ is simple, $W_{j}$ is only the trivial representation,
but we are also interested in the case where the algebraic group $\mathbf{G}$
is simple, but the group $G=\mathbf{G}\left(F\right)$ is not, e.g.,
$\mathbf{G}=\PGL_{d+1}$. In this case, which is explored in \S\ref{sub:Geodesic-flow},
there can be trivial eigenvalues arising from non-trivial representations.
Every nontrivial eigenvalue $\lambda$ of $T$ on $X$ is obtained
as an eigenvalue of $\mathcal{A}$ acting on some eigenfunction $\left(f_{\sigma}\right)_{\sigma\in\Sigma}\in V_{i}^{\Sigma}$.
Let us fix $\xi\in\Sigma$ for which $f_{\xi}\neq0$, and for every
$\sigma\in\Sigma$ denote by $\varphi_{\sigma}:G\rightarrow\mathbb{C}$
the matrix coefficient 
\[
\varphi_{\sigma}\left(g\right)=\left\langle gf_{\sigma},f_{\xi}\right\rangle =\int_{\Gamma\backslash G}f_{\sigma}\left(\Gamma xg\right)\overline{f_{\xi}\left(\Gamma x\right)}d\left(\Gamma x\right)\,.
\]
Since $f_{\sigma}$ is $P_{\sigma}$-invariant so is $\varphi_{\sigma}$,
which means that it defines a function on $G\sigma$. We can patch
these functions together to form a function on $\fC$, namely $\varphi\left(g\sigma\right)=\varphi_{\sigma}\left(g\right)$
(for any $g\in G$ and $\sigma\in\Sigma$). The function $\varphi$
is a $\lambda$-eigenfunction of $T$ on $\fC$ since 
\begin{align*}
\left(A_{\cY_{T,\fX}}\varphi\right)\left(g\sigma\right) & =\sum_{\tau\in\Sigma}\sum_{s\in S^{\sigma,\tau}}\varphi\left(gs\tau\right)=\sum_{\tau\in\Sigma}\sum_{s\in S^{\sigma,\tau}}\varphi_{\tau}\left(gs\right)=\sum_{\tau\in\Sigma}\sum_{s\in S^{\sigma,\tau}}\left\langle gsf_{\tau},f_{\xi}\right\rangle \\
 & =\biggl\langle\sum_{\tau\in\Sigma}\mathcal{A}_{\sigma,\tau}f_{\tau},g^{-1}f_{\xi}\biggr\rangle=\left\langle \lambda f_{\sigma},g^{-1}f_{\xi}\right\rangle =\lambda\varphi_{\sigma}\left(g\right)=\lambda\varphi\left(g\sigma\right)\,,
\end{align*}
and it is not zero since $\varphi\left(\xi\right)=\varphi_{\xi}\left(e\right)=\left\Vert f_{\xi}\right\Vert ^{2}\neq0$.

It is left to show that $\varphi$ is in $L^{2+\varepsilon}\left(\fC\right)$
for all $\varepsilon>0$, and for this we need the Ramanujan assumption~\eqref{eq:ram-def}.
Fix a good maximal compact subgroup $K$ of $G$, and let $d_{\sigma}=\dim\mathrm{span}\left(Kf_{\sigma}\right)$.
Observe that $d_{\sigma}\leq\left[K:P_{\sigma}\right]<\infty$ since
$P_{\sigma}$ is open. Since $V_{i}$ is $\mathcal{I}$-spherical
(as $\mathcal{I}\leq P_{\sigma}$ for all $\sigma$) and infinite-dimensional,
by~\eqref{eq:ram-def} it is weakly-contained in the regular representation
of $G$. Thus, by~\cite[Theorem~2]{Haagerup1988},
\[
\left|\varphi_{\sigma}\left(g\right)\right|=\left|\left\langle gf_{\sigma},f_{\xi}\right\rangle \right|\leq\sqrt{d_{\sigma}d_{\xi}}\left\Vert f_{\sigma}\right\Vert \left\Vert f_{\xi}\right\Vert \Xi\left(g\right)\,,
\]
where $\Xi$ is the Harish--Chandra function, which is in $L^{2+\varepsilon}\left(G\right)$
for all $\varepsilon>0$ (loc.\ cit.). In particular, 
\begin{align*}
\left\Vert \varphi\right\Vert _{L^{2+\varepsilon}\left(\fC\right)}^{2+\varepsilon} & =\sum_{\sigma\in\Sigma}\sum_{g\in\nicefrac{G}{P_{\sigma}}}\left|\varphi_{\sigma}\left(g\right)\right|^{2+\varepsilon}=\sum_{\sigma\in\Sigma}\frac{1}{\mu\left(P_{\sigma}\right)}\int_{G}\left|\varphi_{\sigma}\left(g\right)\right|^{2+\varepsilon}dg\\
 & \leq\left(\sqrt{d_{\sigma}d_{\xi}}\left\Vert f_{\sigma}\right\Vert \left\Vert f_{\xi}\right\Vert \left\Vert \Xi\right\Vert _{2+\varepsilon}\right)^{2+\varepsilon}\sum_{\sigma\in\Sigma}\frac{1}{\mu\left(P_{\sigma}\right)}<\infty
\end{align*}
for every $\varepsilon>0$, so that $\lambda$ is in the $2^{+}$-point
spectrum of $T$ acting on $\fC$. \end{proof}

Before we move on to the trivial eigenvalues, we recall some terminology and facts regarding affine
buildings (cf.~\cite{Brown1989,Garrett1997}, for example).
Let $\mathcal{I}$ be the stabilizer of a fixed chamber $\sigma_{0}$
in $\mathcal{B}$, $\mathfrak{A}$ be an apartment which contains
$\sigma_{0}$, and $\mathcal{N}$ the stabilizer of $\mathfrak{A}$.
Let $\left(\mathcal{W},S\right)$ be the Coxeter system corresponding
to $\left(\mathfrak{A},\sigma_{0}\right)$, i.e., $\mathcal{W}=\mathcal{N}/\left(\mathcal{N}\cap\mathcal{I}\right)=\left\langle S\right\rangle $,
where $S$ is the set of reflections in the walls of $\sigma_{0}$.
When $G$ is simple, the affine Weyl group $\mathcal{W}$ acts simply-transitive
on the chambers in $\mathfrak{A}$. Letting $C_{g}$ denote the double
coset $\mathcal{I}g\mathcal{I}$, $G$ decomposes as a disjoint union
$G=\coprod_{w\in\mathcal{W}}C_{w}$, the so-called Iwahori--Bruhat
decomposition. Furthermore, $\mathcal{W}$ decomposes as $\mathcal{W}_{\mathrm{tr}}\rtimes W_{G}$
where $W_{G}$ is the (finite) spherical Weyl group of $G$, and $\mathcal{W}_{\mathrm{tr}}$
is the coroot lattice, which consists of all translation elements
in $\mathcal{W}$. 

\begin{proposition}\label{prop:T-collision-free-Ramanujan}Let $T$ denote
a collision-free, $G$-equivariant, $k$-regular branching operator
on $\fC\subseteq\mathcal{B}$, let $X=\Gamma\backslash\mathcal{B}$
be a Ramanujan complex, and let $\fX=\Gamma\backslash\fC$. 
Then the trivial eigenvalues of $T$ on $\fX$ all have modulus $k$.
\end{proposition}
\begin{proof}
Recall that the trivial eigenvalues of $T$ on $\fX$  correspond to $\bigcup_{j}\mathrm{spec}\bigl(\mathcal{A}\big|_{W_{j}^{\Sigma}}\bigr)$,
in the notation of the previous proof. Each $W_{j}$ is a one-dimensional
representation of $G$, so that $\mathcal{A}$ acts on $W_{j}^{\Sigma}$
by a $\left|\Sigma\right|\times\left|\Sigma\right|$-matrix. At this
point we assume that $G$ is simple, by restricting it to a finite index subgroup
if necessary, and enlarging $\Sigma$ accordingly. This incurs no
loss of generality, as $T$ is equivariant with respect to the smaller
group as well. Now, in the decomposition of $L^{2}\left(\Gamma\backslash\mathfrak{C}\right)$
in \eqref{eq:L2(X)-decomposition}, each $W_{j}$ is the trivial representation,
thus 
\[
\left(\mathcal{A}\big|_{W_{j}^{\Sigma}}\right)_{\sigma,\tau}=\#\left\{ \tau'\in T\left(\sigma\right)\,\middle|\,\tau'\in G\tau\right\} \,.
\]
This is precisely the adjacency matrix of $\cY_{T,\Sigma}$, the digraph
corresponding to $T$ acting on $\Sigma\simeq G\backslash\mathfrak{C}$
(which can have loops and multiple edges). Thus, the trivial spectrum
of $T$ on $\mathfrak{X}$ is the spectrum of $T$ on $\Sigma$. We
will show that each vertex in $\cY_{T,\Sigma}$ has a unique out-neighbor,
repeating with multiplicity $k$. This implies that $\cY_{T,\Sigma}$
is a disjoint union of cycles, each repeating with multiplicity $k$.
It follows that the eigenvalues of $\cY_{T,\Sigma}$ are all of modulus
$k$.

Assume to the contrary that some $\sigma\in\Sigma$ has two different
out-neighbors $\tau_{1},\tau_{2}$ in $\cY_{T,\Sigma}$. This means
that in $\mathfrak{X}$ the set $T\left(\sigma\right)$ contains cells
$\tau'_{1},\tau'_{2}$ from the distinct orbits $G\tau_{1}$ and $G\tau_{2}$.
Since $\cY_{T,\Sigma}$ is Eulerian (the in-degree of each of vertices
equals its out-degree), there exists a (directed) cycle $C_{1}$ in
$\cY_{T,\Sigma}$ which begins with the edge $\sigma\rightarrow\tau_{1}$
(and ends at $\sigma$). Since $\cY_{T,\Sigma}\setminus E\left(C_{1}\right)$
is still Eulerian, there is also a cycle $C_{2}$ which begins with
$\sigma\rightarrow\tau_{2}$. Thus, for $m_{i}=\mathrm{length}\left(C_{i}\right)$
($i=1,2$) there exist $g_{1},g_{2}\in G$ such that $g_{i}\sigma\in T^{m_{i}-1}\left(\tau'_{i}\right)\subseteq T^{m_{i}}\left(\sigma\right)$.
Write $g_{i}=b_{i}w_{i}b_{i}'$ with $b_{i},b_{i}'\in\mathcal{I}$
and $w_{i}\in\mathcal{W}$, using the Iwahori--Bruhat decomposition.
Since $I \leq P_{\sigma}$, $b_{i}w_{i}\sigma=b_{i}w_{i}b_{i}'\sigma\in T^{m_{i}-1}\left(\tau'_{i}\right)$,
which implies that $w_{i}\sigma\in T^{m_{i}-1}\left(b_{i}^{-1}\tau'_{i}\right)$.
Note that $b_{1}^{-1}\tau'_{1}\neq b_{2}^{-1}\tau'_{2}$, since they
are not in the same $G$-orbit, and both are descendants of $\sigma$
since $b_{i}^{-1}\tau'_{i}\in T\left(b_{i}^{-1}\sigma\right)=T\left(\sigma\right)$.
Therefore, we must have $w_{1}\sigma\neq w_{2}\sigma$ by the collision-free
assumption. Taking $m=\left|W_{G}\right|$, we have $w_{1}^{m}w_{2}^{m}\sigma=w_{2}^{m}w_{1}^{m}\sigma$,
since $w_{i}^{m}$ lie in the abelian group $\mathcal{W}_{tr}$, and
we now show that this is a common descendant of $w_{1}\sigma$ and
$w_{2}\sigma$, arriving at a contradiction. Indeed, since 
$w_{i}\sigma=b_{i}^{-1}g_{i}\sigma\in T^{m_{i}}\left(b_{i}^{-1}\sigma\right)=T^{m_{i}}\left(\sigma\right)$
we have 
\begin{align*}
w_{1}^{m}w_{2}^{m}\sigma & \in w_{1}^{m}w_{2}^{m-1}T^{m_{2}}\left(\sigma\right)=w_{1}^{m}w_{2}^{m-2}T^{m_{2}}\left(w_{2}\sigma\right)\subseteq w_{1}^{m}w_{2}^{m-2}T^{2m_{2}}\left(\sigma\right)=\ldots\\
\ldots & \subseteq w_{1}^{m_{1}}T^{mm_{2}}\left(\sigma\right)=\ldots=w_{1}T^{\left(m-1\right)m_{1}+mm_{2}}\left(\sigma\right)=T^{\left(m-1\right)m_{1}+mm_{2}}\left(w_{1}\sigma\right)\,,
\end{align*}
and in a similar manner $w_{2}^{m}w_{1}^{m}\sigma\in T^{\left(m-1\right)m_{2}+mm_{1}}\left(w_{2}\sigma\right)$.\end{proof}

We next wish to show that when $T$ acts on $L^{2}\left(\fX\right)$,
it is $r$-normal for some $r$, depending only on $G$. To this end,
we need the following proposition. \begin{proposition}[{cf.~\cite[Prop. 2.6]{casselman1980unramified}}]
Every unitary irreducible representation $V$ of $G$ with $\mathcal{I}$-fixed
vectors can be embedded in a principal series representation. \end{proposition}
Each principal series representation of $G$ is obtained by induction
from a character $\chi$ of a minimal parabolic subgroup $B$ (a Borel
subgroup): 
\[
\mathrm{Ind}_{B}^{G}\chi=\left\{ f:G\rightarrow\mathbb{C}\,\middle|\,f\left(bg\right)=\chi\left(b\right)f\left(g\right)\quad\forall b\in B,g\in G\right\} \,.
\]

Thus, if $f\in V^{\mathcal{I}}\hookrightarrow\bigl(\mathrm{Ind}_{B}^{G}\chi\bigr)^{\mathcal{I}}$
then $f$ satisfies $f\left(bgi\right)=\chi\left(b\right)f\left(g\right)$
for every $b\in B,g\in G,i\in\mathcal{I}$, and the decomposition
$G=\coprod_{w\in W_{G}}Bw\mathcal{I}$ (cf.~\cite[(19)]{casselman1980unramified})
shows that $\dim V^{\mathcal{I}}\leq\left|W_{G}\right|$. Thus, each
summand in the decomposition of $L^{2}\left(\fX\right)$
in~\eqref{eq:L2(X)-decomposition} is of dimension at most $\left|\Sigma\right|\left|W_{G}\right|$.
Since $T$ acts on $L^{2}\left(\fX\right)$ by an
element in the group algebra of $G$, it decomposes with respect to~\eqref{eq:L2(X)-decomposition},
and therefore it is $\left(\left|\Sigma\right|\left|W_{G}\right|\right)$-normal.
In particular, since $\left|\Sigma\right|\leq2^{\rank G+1}$ we arrive at the following.

\begin{proposition}\label{prop:YT-collision-free-r-normal}If $T$
is a $G$-equivariant branching operator on $\fC\subset\mathcal{B}$,
and $X=\Gamma\backslash\mathcal{B}$ is a finite quotient of $\mathcal{B}$,
then $A_{\cY_{T,\fX}}$ is $r$-normal for some $r=r_G$.\end{proposition}

\begin{proof}[\textbf{\emph{Proof of Theorem~\ref{thm:T-collision-free-r-normal-Ramanujan}}}]
By Proposition \ref{prop:Ramanujan-L2e}, every nontrivial
eigenvalue $\lambda$ of $T$ on $\mathfrak{X}$ is also in the $2^{+}$-point
spectrum of $T$ on $\mathfrak{C}$, so $\left|\lambda\right|\leq\sqrt{k}$ by Proposition~\ref{prop:collision-free-riemann}.
The trivial eigenvalues all have modulus $k$ by Proposition~\ref{prop:T-collision-free-Ramanujan}, thus $\cY_{T,\fX}$ is Ramanujan. 
Finally, Proposition~\ref{prop:YT-collision-free-r-normal} shows that $\cY_{T,\fX}$ is $r$-normal for some $r=r_G$. 
\end{proof}

\begin{proof}[\textbf{\emph{Proof of Theorem~\ref{thm:Ramanujan-cutoff}}}]
By Theorem~\ref{thm:T-collision-free-r-normal-Ramanujan}, the finite digraph
$\cY=\cY_{T,\fX}$ is Ramanujan as well as  $r$-normal for some $r=r_G$. It is also $k$-regular, being the quotient of the $k$-regular digraph $\cY_{T,\fC}$, and thus its stationary distribution $\pi$ is uniform.
Further note that the hypothesis $\max_{x,y}\rho(x,y)<\infty$ in Theorem~\ref{thm:Ramanujan-cutoff} says that $\cY$ is strongly-connected.
Finally, to bound the period of $\cY$, recall that all of its nontrivial
eigenvalues have modulus at most $\sqrt{k}$ (being Ramanujan),
and hence its period is bounded by the number of trivial eigenvalues (those
with modulus $k$), which is at most
$  |\nicefrac{G}{G'}||\Sigma|\leq|\nicefrac{G}{G'}|2^{\rank G+1}$, where $G'$ is the derived subgroup of $G$.

We may thus apply Proposition~\ref{prop:cutoff-digraph} to $\cY$ (e.g., taking $M:=\big(|\nicefrac{G}{G'}|2^{\rank G + 1}\big)!$ to be divisible by its period), establishing part~(i) of Theorem~\ref{thm:Ramanujan-cutoff}, whereas part~(ii) is an immediate consequence of it as in the proof of Theorem~\ref{thm:distances} in~\S\ref{sec:Random-walks-digraphs}.
\end{proof}

\begin{remark} Another way to explain the results of this section
is by considering $\mathcal{H=\mathcal{H}}\left(G,\mathcal{I}\right)$,
the Iwahori--Hecke algebra of $G$. This is the algebra
of all complex bi-$\mathcal{I}$-invariant functions on $G$ with
compact support, w.r.t.\ convolution. As $G$ acts on $L^{2}\left(\Gamma\backslash G\right)$
(from the right), $\mathcal{H}$ acts on $L^{2}\left(\Gamma\backslash G/\mathcal{I}\right)$
--- the space of complex functions on the chambers of $X$.

Every irreducible $G$-subrepresentation $U$ of $L^{2}\left(\Gamma\backslash G\right)$
with $U^{\mathcal{I}}\neq\left\{ 0\right\} ,$ induces an irreducible
representation of the Iwahori--Hecke algebra $\mathcal{H}$
on $U^{\mathcal{I}}$ which is a subspace of $L^{2}\left(\Gamma\backslash G\right)^{\mathcal{I}}=L^{2}\left(\Gamma\backslash G/\mathcal{I}\right)=L^{2}\left(X\left(d\right)\right)$,
where $X\left(d\right)$ is the set of $d$-dimensional cells of $X=\Gamma\backslash\mathcal{B}$.
Condition~\eqref{eq:ram-def} enables us to deduce that for operators
$T\in\mathcal{H}$, the nontrivial eigenvalues of $T$ acting on $X$
appear in the spectrum of $T$ acting on $L^{2}\left(\mathcal{B}(d)\right)$.
The irreducible representations of $\mathcal{H}$ are all of bounded
dimensions, which ensures that $T$ decomposes (w.r.t.\ a suitable
orthonormal basis) as a sum of blocks of bounded size. For a thorough
treatment of this approach, we refer the reader to~\cite{first2016ramanujan,Kamber2017LP}.
\end{remark}
  
\section{Collision-free operators on affine buildings}
\label{sec:col-free}

In this section we study collision-free branching operators on buildings.
In~\S\ref{sub:Geodesic-flow} we focus on the buildings associated
with $\PGL_{m}$, and explore a family of operators which originate
in the works of Li and Kang~\cite{kang2014zeta,kang2016riemann},
and can be thought of as non-Archimedean geodesic flows. In~\S\ref{sub:branching-on-simple-gps},
we give, for a general simple group $G$, a necessary and sufficient condition for a $G$-equivariant branching operator
on the chambers of the building of $G$, to be collision-free. 

\subsection{Geodesic flows on complexes of type $\widetilde{A}_{d}$}\label{sub:Geodesic-flow}

Let $F$ be a non-Archimedean local field with ring of integer $\mathcal{O}$,
uniformizer $\pi$, and residue field $\nicefrac{\mathcal{O}}{\pi\mathcal{O}}$
of order $q$. Let $G=\PGL_{m}\left(F\right)$, and fix the maximal
compact subgroup $K=\PGL_{m}\left(\mathcal{O}\right)$. The building
$\mathcal{B}=\mathcal{B}\left(\PGL_{m}\left(F\right)\right)$ is a
simplicial complex of dimension $d=m-1$, whose vertices correspond
to the $G/K$-cosets, and they are ``colored'' by the function 
\[
\col:\nicefrac{G}{K}\rightarrow\nicefrac{\mathbb{Z}}{m\mathbb{Z}},\qquad\col\left(gK\right)\equiv\mathrm{ord}_{\pi}\left(\det\left(g\right)\right)\Mod{m}\,.
\]
Each $g\in G$ corresponds to a homothety class of $\mathcal{O}$-lattices
in $F^{m}$, namely $g\mathcal{O}^{m}$. The vertices $g_{0}K,\ldots,g_{j}K$
form a cell in $\mathcal{B}$ if and only if each homothety class
$g_{i}\mathcal{O}^{m}$ can be represented by a lattice $L_{i}$,
so that, possibly after reordering, 
\[
L_{0}>L_{1}>\ldots>L_{j}>\pi L_{0}\,.
\]
The edges of $\cB$ are also colored: if $e=(x,y)$, then $\col(e)=\col(y)-\col(x)$ (as an element of the group $\Z/m\Z$). Note that, while the color of vertices is not $G$-invariant, the color of edges is.

The geodesic flow on $\mathcal{B}$ which we now define is a simplicial
analog of the geodesic flow on the unit bundle of a manifold, where
the role of the ``direction'' vector is played by a cell. Formally,
for every $1\leq j\leq m-1$, we define the ``$j$-th-unit bundle''
$UT^{j}\mathcal{B}$ as the set of pairs $\left(v\in \mathcal{B}\left(0\right),\sigma\in\mathcal{B}\left(j\right)\right)$
(where $\cB(j)$ is the set of $j$-dimensional cells) such that $v\in\sigma$, and the vertices in $\sigma$ are of colors
$\col\left(v\right),\col\left(v\right)+1,\ldots,\col\left(v\right)+j$.
In the next step of the flow, the basepoint $v$ is replaced by the
``next'' vertex in $\sigma$ (the one with color $\col\left(v\right)+1$).
A more difficult question is what happens to the ``direction''
$\sigma$, as there is no notion of parallel transport in the discrete
settings. In fact, there are several possible options, which is
why we obtain a branching operator, and not a deterministic flow.
\begin{definition}\label{def:geodesic-flow} The geodesic flow $T$
on $UT^{j}\mathcal{B}$ is defined as follows: $T(\sigma)$ for $\sigma=\left(v_{0},\left\{ v_{0},\ldots,v_{j}\right\} \right)\in UT^{j}\mathcal{B}$,
with $\col\left(v_{i}\right)\equiv\col\left(v_{0}\right)+i$, consists
of all the cells $\sigma'=\left(v_{1},\left\{ v_{1},\ldots,v_{j},w\right\} \right)\in UT^{j}\mathcal{B}$
such that 
\begin{enumerate}
\item $\col\left(w\right)\equiv\col\left(v_{1}\right)+j$ (the ``direction
vector'' is based at $v_{1}$), 
\item $\{v_{0},v_{1},\ldots,v_{j},w\}$ is not a cell (geodicity). 
\end{enumerate}
\end{definition} Note that $T$ is $G$-equivariant, so it defines
a branching operator on any quotient $X=\Gamma\backslash\mathcal{B}$.
For example, in dimension one, the geodesic flow on $UT^{1}X$ coincides
with the non-backtracking walk on the graph $X$, since any two neighbors
$v,w$ satisfy $\col\left(w\right)\equiv\col\left(v\right)+1$, and
there are no triangles, so (2) reduces to $w\neq v_{0}$. When
$X$ is of higher dimension, (2) also prevents the flow on $U T^{1}X$
from making two steps along the edges of a triangle, as there is a
``shorter route''.

Let us analyze the geodesic flow in the language of Section \ref{sec:Ramanujan-complexes}.
Denote by $v^{\left(j\right)}$ the vertex 
\[
v^{\left(j\right)}=\diag\left(1^{\times j},\pi^{\times\left(m-j\right)}\right)\cdot K\,,
\]
and by $\sigma^{\left(j\right)}$ the ordered cell 
\[
\sigma^{\left(j\right)}=\left[v^{\left(j\right)},v^{\left(j-1\right)}\ldots,v^{\left(1\right)},v^{\left(0\right)}\right]\,.
\]
Since $G$ acts transitively on $UT^{j}\mathcal{B}$, and the stabilizer
of a pair $\left(v,\sigma\right)$ is the point-wise stabilizer of
$\sigma$, we can identify the pair $\left(v,\sigma\right)$ with
the unique ordering $\left[v_{0},\ldots,v_{j}\right]$ of $\sigma$
with consecutive colors and $v_{0}=v$. Furthermore, if $P_{j}$ is
the point-wise stabilizer of $\sigma^{\left(j\right)}$, then $gP_{j}\mapsto\left(gv^{\left(j\right)},g\sigma^{\left(j\right)}\right)$
identifies $G/P_{j}$ with $UT^{j}\mathcal{B}$. Since $v^{\left(0\right)}$
corresponds to $\mathcal{O}^{m}$, the cells in $T\left(\sigma^{\left(j\right)}\right)$
are of the form $\left[v^{\left(j-1\right)}\ldots,v^{\left(1\right)},v^{\left(0\right)},w\right]$
with $w$ corresponding to a maximal sublattice of $\mathcal{O}^{m}$
which:\\
 (1) Contains $\diag\left(1^{\times\left(j-1\right)},\pi^{\times\left(m-j+1\right)}\right)\cdot\mathcal{O}^{m}$,
so that $\left[v^{\left(j-1\right)},\ldots,v^{\left(0\right)},w\right]$
is a cell.\\
 (2) Does not contain $\diag\left(1^{\times j},\pi^{\times\left(m-j\right)}\right)\cdot\mathcal{O}^{m}$,
so that the flow is geodesic.\\
 The sublattices which satisfy these two conditions are those of the
form 
\[
L_{a_{1},\ldots,a_{m-j}}=\left(\begin{array}{c|c|c}
~\Large{\mbox{\ensuremath{I_{j-1}}}}\vphantom{\Big)}~ & \\
\hline  & \pi & a_{1}\:\cdots\:a_{m-j}\\
\hline  &  & ~\Large{\text{\ensuremath{I_{m-j}}}}\vphantom{\Big)}\ 
\end{array}\right)\cdot\mathcal{O}^{m}\qquad\left(a_{i}\in\nicefrac{\mathcal{O}}{\pi\mathcal{O}}\cong\mathbb{F}_{q}\right)\,,
\]
which in particular shows that $T$ is $q^{m-j}$-out-regular. 
As $G$ acts transitively on $U T^j\cB$, it is in fact $q^{m-j}$-regular.
To
see that it is also collision-free, observe that the matrix 
\[
g_{a_{1},\ldots,a_{m-j}}=\left(\begin{array}{c|c|c}
 & ~\Large{\mbox{\ensuremath{I_{j-1}}}}\vphantom{\Big)}~\\
\hline \pi &  & a_{1}\:\cdots\:a_{m-j}\\
\hline  &  & ~\Large{\text{\ensuremath{I_{m-j}}}}\vphantom{\Big)}\ 
\end{array}\right)
\]
takes $v^{\left(i\right)}$ to $v^{\left(i-1\right)}$ for $1\leq i\leq j$,
and $v^{\left(0\right)}$ to the vertex corresponding to $L_{a_{1},\ldots,a_{m-j}}$.
In particular, this implies that $P_{j}g_{0,\ldots,0}=\left\{ g_{a_{1},\ldots,a_{m-j}}\,\middle|\,a_{i}\in\mathbb{F}_{q}\right\} $,
and it is now an easy exercise to verify that for any $\ell\in\mathbb{N}$
\[
\left(P_{j}g_{0,\ldots,0}\right)^{j\ell}=\left\{ \left(\begin{array}{ccc|c}
\pi^{\ell} &  &  & a_{1,1}\ \cdots\ a_{1,m-j}\\
 & \ddots &  & \vdots\quad\ddots\quad\ \vdots\ \\
 &  & \pi^{\ell} & a_{j,1}\ \cdots\ a_{j,m-j}\\
\hline  &  &  & ~\Large{\text{\ensuremath{I_{m-j}}}}\vphantom{\Big)}\ 
\end{array}\right)\;\middle|\;a_{x,y}\in\nicefrac{\mathcal{O}}{\pi^{\ell}\mathcal{O}}\right\} \,.
\]
Since these matrices carry $\mathcal{O}^{m}$ to non-equivalent lattices,
the last vertices in $T^{\ell j}\bigl(g\sigma^{\left(j\right)}\bigr)=g\left(P_{j}g_{0,\ldots,0}\right)^{\ell j}\sigma^{\left(j\right)}$
are all different, implying that $T^{j}$ is collision-free,
and hence so is $T$. We have thus shown: \begin{proposition} The $j$-dimensional
geodesic flow on $\mathcal{B}\left(\PGL_{m}\right)$ is collision-free.
\end{proposition}

For the geodesic flow, we can give an explicit
bound for the normality of the adjacency operator, and for $j=1$
we show below that it is optimal. 

\begin{proposition} \label{prop:j-flow-(m)j-normal}The
$j$-dimensional geodesic flow on finite quotients of $\mathcal{B}\left(\PGL_{m}\right)$
is $\left(m\right)_{j}$-normal, where $\left(m\right)_{j}=m!/(m-j)!$.
\end{proposition} \begin{proof} Let $V$ be an irreducible representation
of $G=\PGL_{m}\left(F\right)$ with $V^{P_{j}}\neq0$. Let $B$ be
the standard Borel, consisting of all upper-triangular matrices in
$G$. By~\cite[Prop. 2.6]{casselman1980unramified} there exists
a character $\chi$ of $B$, which is trivial on $B\cap K$, and 
\begin{equation}
V^{P_{j}}\hookrightarrow\left(\mathrm{Ind}_{B}^{G}\chi\right)^{P_{j}}=\left\{ f:G\rightarrow\mathbb{C}\,\middle|\;f\left(bgp\right)=\chi\left(b\right)f\left(g\right)\quad\forall b\in B,p\in P_{j}\right\} \,.\label{eq:prin-ser}
\end{equation}
We will show that $\dim\bigl(\mathrm{Ind}_{B}^{G}\chi\bigr)^{P_{j}}=\left(m\right)_{j}$.
Denote by $W$ be the Weyl group of $G$, which consists of all permutation
matrices, and let $W_{j}=W\cap P_{j}$. Observe that 
\begin{align}
P_{j} & =\bigcap_{i=0}^{j}\diag\left(1^{\times i},\pi^{\times\left(m-i\right)}\right)\cdot K\cdot\diag\left(1^{\times i},\left(\nicefrac{1}{\pi}\right)^{\times\left(m-i\right)}\right)\nonumber \\
 & =\left\{ g\in K\,\middle|\,g_{r,c}\in\pi\mathcal{O}\text{ for }c\leq\min\left(j,r-1\right)\right\} \,,\label{eq:P_j}
\end{align}
which implies that $W_{j}=\mathrm{Sym}_{\left\{ j+1,\ldots,m\right\} }$.
If $w_{1},\ldots,w_{\left(m\right)_{j}}$ is a transversal for $W/W_{j}$,
we claim that $G=\coprod_{i=1}^{\left(m\right)_{j}}Bw_{i}P_{j}$.
Indeed, the decomposition $G=\coprod_{w\in W}Bw\mathcal{I}$
already implies that $G=\bigcup_{i}Bw_{i}P_{j}$, and we need to show
these are disjoint. If $w_{i}=bw_{k}p$ for some $b\in B$, $p\in P_{j}$,
then $b^{-1}=w_{k}pw_{i}^{-1}\in B\cap w_{k}Pw_{i}^{-1}\subseteq B\cap K$,
which implies that the diagonal entries of $b^{-1}$ are in $\mathcal{O}^{\times}$.
On the other hand, let $r>m-j$, and observe that $p_{r,c}=\left(w_{k}^{-1}b^{-1}w_{i}\right)_{r,c}=\left(b^{-1}\right)_{w_{k}\left(r\right),w_{i}\left(c\right)}$,
so that 
\[ p_{w_{k}^{-1}\left(w_{i}\left(c\right)\right),c}=\left(b^{-1}\right)_{w_{i}\left(c\right),w_{i}\left(c\right)}\in\mathcal{O}^{\times}\,.\]
By~\eqref{eq:P_j}, this implies $w_{k}^{-1}\left(w_{i}\left(c\right)\right)\leq c$
for $1\leq c\leq j$, and thus $w_{i}\left(c\right)=w_{k}\left(c\right)$
for $1\leq c\leq j$, hence $w_{i}W_{j}=w_{k}W_{j}$. This implies
$i=k$, and thus $G=\coprod_{i=1}^{\left(m\right)_{j}}Bw_{i}P_{j}$
as claimed. In addition, since $B\cap w_{i}P_{j}w_{i}^{-1}\subseteq B\cap K$
and $\chi\big|_{B\cap K}\equiv1$ the functions 
\[
f_{i'}\left(bw_{i}p\right)=\chi\left(b\right)\delta_{i,i'}\quad\left(b\in B,i,i'\in\left\{ 1,\ldots,\left(m\right)_{j}\right\} ,p\in P_{j}\right)
\]
are well defined, and thus form a basis for $\left(\mathrm{Ind}_{B}^{G}\chi\right)^{P_{j}}$.
\end{proof} 
\begin{proposition} \label{prop:flow-not-m-1-normal}
The one-dimensional geodesic flow
on Ramanujan quotients of $\mathcal{B}\left(\PGL_{m}\right)$ is not
$r$-normal for any $r<m$. \end{proposition}
 \begin{proof} Let
$z_{1},\ldots,z_{m}\in\mathbb{C}$ with $\left|z_{i}\right|=1$ for
all $i$ and $\prod_{i=1}^{m}z_{i}=1$, and let 
\[
\chi:B\rightarrow\mathbb{C}^{\times},\qquad\chi\left(\left(b_{ij}\right)\right)=\prod_{i=1}^{m}\left(q^{i-\frac{m+1}{2}}z_{i}\right)^{\mathrm{ord}_{\pi}\left(b_{ii}\right)}\,.
\]
The representation $\mathrm{Ind}_{B}^{G}\chi$ is irreducible, unitary
and $K$-spherical. Furthermore, if $X=\Gamma\backslash\mathcal{B}$
is Ramanujan, then every $K$-spherical infinite-dimensional subrepresentation
of $L^{2}\left(\Gamma\backslash G\right)$ is of this form (cf.~\cite{Lubotzky2005a}).
Since $L^{2}\left(X\left(0\right)\right)\cong L^{2}\left(\Gamma\backslash G\right)^{K}$ and
$\dim\left(\mathrm{Ind}_{B}^{G}\chi\right)^{K}=1$, the number of such representations which appear in $L^{2}\left(\Gamma\backslash G\right)$
is the number of vertices in $\Gamma\backslash\mathcal{B}$, minus
the number of one-dimensional subrepresentations of $L^{2}\left(\Gamma\backslash G\right)$,
which is bounded by $m$. By the proof of Proposition \ref{prop:j-flow-(m)j-normal},
$\left(\mathrm{Ind}_{B}^{G}\chi\right)^{P_{1}}$ is $m$-dimensional,
with basis $f_{1},\ldots,f_{m}$ defined by 
\[
f_{i}\left(bwp\right)=\chi\left(b\right)\delta_{w\left(1\right),i}\qquad\left(b\in B,w\in S_{m},p\in P_{1}\right)\,.
\]
This is an orthogonal basis, albeit not orthonormal ---
one can verify that $\left\Vert f_{i}\right\Vert =q^{i/2}$. Recall
that the geodesic flow acts by $\mathcal{A}=\sum_{\vec{a}\in\mathbb{F}_{q}^{m-1}}g_{\vec{a}}$,
where $g_{\vec{a}}$ is the identity matrix with the first row replaced
by $\left(\pi,a_{1},a_{2},\ldots,a_{m-1}\right)$. Since $g_{\vec{a}}\in B$,
\[
\left(\mathcal{A}f_{i}\right)\left(e\right)=\sum_{\vec{a}\in\mathbb{F}_{q}^{m-1}}f_{i}\left(g_{\vec{a}}\right)=\sum_{\vec{a}\in\mathbb{F}_{q}^{m-1}}\chi\left(g_{\vec{a}}\right)=q^{m-1}\cdot q^{1-\frac{m+1}{2}}z_{1}\delta_{1,i}=\delta_{1,i}q^{\frac{m-1}{2}}z_{1}\,.
\]
Assume from now on $1<j\leq m$, and observe $\left(\mathcal{A}f_{i}\right)\left(\left(1\,j\right)\right)=\sum_{\vec{a}\in\mathbb{F}_{q}^{m-1}}f_{i}\left(\left(1\,j\right)g_{\vec{a}}\right)$.
If $a_{1}=\ldots=a_{j-1}=0$, and $\pi_{j}$ is the identity matrix
with the $\left(j,j\right)$-th entry replaced by $\pi$, then $u:=\pi_{j}^{-1}\left(1\,j\right)g_{\vec{a}}\left(1\,j\right)\in U$,
the unipotent upper-triangular matrices in $G$. Since $\chi\big|_{U}\equiv1$,
\[
\sum_{\vec{a}\in\mathbb{F}_{q}^{m-1}\,:\;a_{1}=\ldots=a_{j-1}=0}\negthickspace\negthickspace\negthickspace\negthickspace\negthickspace\negthickspace\negthickspace\negthickspace\negthickspace\negthickspace\negthickspace\negthickspace f_{i}\left(\left(1\,j\right)g_{\vec{a}}\right)=\sum_{\ldots}f_{i}\left(\pi_{j}u\left(1\,j\right)\right)=q^{m-j}\chi\left(\pi_{j}\right)\delta_{j,i}=\delta_{j,i}q^{\frac{m-1}{2}}z_{j}\,.
\]
Next, if $a_{j-1}\neq0$, let
 $ \vec{b}=\left(-\tfrac{a_{1}}{a_{j-1}},\ldots,-\tfrac{a_{j-2}}{a_{j-1}},\tfrac{1}{a_{j-1}},-\tfrac{a_{j}}{a_{j-1}},\ldots,-\tfrac{a_{m-1}}{a_{j-1}}\right)\,,
 $
and note that $g_{\vec{b}}^{-1}\left(1\,j\right)g_{\vec{a}}\in P_{1}$,
so that 
\[
\sum_{\vec{a}\in\mathbb{F}_{q}^{m-1}\,:\;a_{j-1}\neq0}\negthickspace\negthickspace\negthickspace\negthickspace\negthickspace\negthickspace f_{i}\left(\left(1\,j\right)g_{\vec{a}}\right)=\left(q-1\right)q^{m-2}\chi\left(g_{\vec{b}}\right)\delta_{1,i}=\delta_{1,i}\left(q-1\right)q^{\frac{m-3}{2}}z_{1}\,.
\]
Finally, if $a_{j-1}=0$, let $k=\min\left\{ i\,\middle|\,a_{i-1}\neq0\right\} $
(note $k<j$), and let $b$ be the identity matrix with the $k$-th
row replaced by 
\[
b_{k,c}=\begin{cases}
0 & c<k\\
\pi & c=k\\
1/a_{k-1} & c=j\\
-a_{c-1}/a_{k-1} & \mbox{otherwise}\,.
\end{cases}
\]
Observe that $p:=\left(1\,k\right)b^{-1}\left(1\,j\right)g_{\vec{a}}\in P_{1}$,
so that
\begin{align*}
\sum_{\vec{a}\in S} f_{i}\left(\left(1\,j\right)g_{\vec{a}}\right) & =\sum_{\vec a \in S}f_{i}\left(\left(1\,k\right)bp\right) =\left(q-1\right)q^{m-k-1}\chi\left(\left(1\,k\right)b\right)=\left(q-1\right)q^{\frac{m-3}{2}}z_{k}\delta_{k,i}
\end{align*}
for $S = \{\vec{a}\in\mathbb{F}_{q}^{m-1}:a_{j}=a_{1}=\ldots=a_{k-1}=0,\,a_{k}\neq0\}$.
Altogether, we obtain that $\mathcal{A}$ acts
on $\left(\mathrm{Ind}_{B}^{G}\chi\right)^{P_{1}}$, with respect
to the orthonormal basis $\left\{ q^{-i/2}f_{i}\right\} $, by 
\begin{equation}
\left(\begin{matrix}q^{\frac{m-1}{2}}z_{1}\\
\left(q-1\right)q^{\frac{m-2}{2}}z_{1} & q^{\frac{m-1}{2}}z_{2}\\
\left(q-1\right)q^{\frac{m-1}{2}}z_{1} & \left(q-1\right)q^{\frac{m-2}{2}}z_{2} & q^{\frac{m-1}{2}}z_{3}\\
\vdots & \vdots & \ddots & \ddots\\
\left(q-1\right)q^{m-2}z_{1} & \left(q-1\right)q^{\frac{2d-5}{2}}z_{2} & \cdots & \left(q-1\right)q^{\frac{m-2}{2}}z_{m-1} & q^{\frac{m-1}{2}}z_{m}
\end{matrix}\right)\,.\label{eq:block}
\end{equation}
Each infinite-dimensional $K$-spherical representation in $L^{2}\left(\Gamma\backslash G\right)$
contributes a block of this form to $A_{\cY_{T,\fX}}$,
and hence, by the uniqueness of the $QR$-decomposition, it is not $r$-normal for any $r<m$. \end{proof} Let us remark that for the case
$m=2$, Ramanujan quotients of the building $\mathcal{B}$ are just
Ramanujan graphs, and the one-dimensional geodesic flow is the non-backtracking
random walk. In this case,~\eqref{eq:block} coincides with the unitary
block decomposition carried out for this walk in~\cite{Lubetzky2016}.

\subsection{Branching operators on chambers for simple groups}

\label{sub:branching-on-simple-gps} In this section we take $G=\mathbf{G}\left(F\right)$
to be any simple group, and determine which $G$-equivariant branching
operators on the chambers of $\mathcal{B}=\mathcal{B}\left(G\right)$
are collision-free. 

Any $G$-equivariant branching operator is of the form $T_{\mathscr{G}}\left(\sigma_{0}\right)=\mathcal{I}\mathscr{G}\sigma_{0}$
for some finite set $\mathscr{G}\subseteq G$. Observe that $T_{\mathscr{G}}\left(\sigma_{0}\right)=T_{\mathcal{I}\mathscr{G}\mathcal{I}}\left(\sigma_{0}\right)$,
so that decomposing $\mathcal{I}\mathscr{G}\mathcal{I}=\cup_{w\in\mathscr{W}}C_{w}$
with $\mathscr{W}\subseteq\mathcal{W}$, $T_{\mathscr{G}}$ equals
$T_{\mathscr{W}}$.

\begin{proposition}If $T=T_{\mathscr{W}}$ is collision-free for
$\mathscr{W}\subseteq\mathcal{W}$, then $\left|\mathscr{W}\right|=1$.\end{proposition}

\begin{proof}Assume that $\mathscr{W}$ contains $w_{0}\neq w_{1}.$
Denoting $\mu=\left|W_{G}\right|$, one has $w_{0}^{\mu}w_{1}^{\mu}=w_{1}^{\mu}w_{0}^{\mu}$ as in 
the proof of Proposition~\ref{prop:T-collision-free-Ramanujan}.
Therefore, $w_{0}^{\mu}w_{1}^{\mu}\left(\sigma_{0}\right)$
is contained both in 
$T^{2\mu-1}\left(w_{0}\sigma_{0}\right)=\left(\mathcal{I}\mathscr{W}\mathcal{I}\right)^{2\mu-1}w_{0}\sigma_{0}$,
and in $T^{2\mu-1}\left(w_{1}\sigma_{0}\right)$. Since $w_{0}\sigma_{0},w_{1}\sigma_{0}\in T\left(\sigma_{0}\right)$,
this implies that $T$ is not collision-free.\end{proof}

Thus, we can restrict our attention to the case where $T=T_{w_{0}}$
for a single $w_{0}\in\mathcal{W}$. For what follows we need the notion
of the Weyl length $\ell\left(w\right)$ of $w\in\mathcal{W}=\left\langle S\right\rangle $, which
is the shortest length of a word in $S$ which equals $w$. In addition,
we need the projection $\rho:\mathcal{B}\rightarrow\mathfrak{A}$ defined by
the Iwahori--Bruhat decomposition,
namely $\rho\big|_{C_{w}\sigma_{0}}\equiv w\sigma_{0}$ $\left(\forall w\in\mathcal{W}\right)$.
Denote $k=\left|T_{w_{0}}\left(\sigma_{0}\right)\right|$, and let $\mu\geq1$
be an integer such that $w_{0}^{\mu}\in\mathcal{W}_{\mathrm{tr}}$
(for example, we can always take $\mu=\left|W_{G}\right|$). \begin{theorem}
\label{thm:bruhat-collision-free} With the above notation, the following
are equivalent: 
\begin{enumerate}
\item The operator $T_{w_{0}}$ is collision-free. 
\item $\left|T_{w_{0}}^{\mu}\left(\sigma_{0}\right)\right|=k^{\mu}$. 
\item $\ell\left(w_{0}^{\mu}\right)=\mu\cdot\ell\left(w_{0}\right).$ 
\item The random walk defined by $T_{w_{0}}$ projects onto a deterministic
walk in $\mathfrak{A}$, namely $\bigl|\rho\bigl(T_{w_{0}}^{j}\left(\sigma_{0}\right)\bigr)\bigr|=1$
for all $j$. 
\end{enumerate}
\end{theorem} Note that if $w_{0}\in\mathcal{W}_{tr}$ then \emph{(3)}
 holds with $\mu=1$, so $T_{w_0}$ is collision-free. 
\begin{proof} It is clear that
\emph{(1)}$\Rightarrow$\emph{(2)}, and we proceed to show that for
every $j\in\mathbb{N}$, 
\begin{equation}
\left|T_{w_{0}}^{j}\left(\sigma_{0}\right)\right|=k^{j}\quad\Leftrightarrow\quad\ell\left(w_{0}^{j}\right)=j\cdot\ell\left(w_{0}\right)\quad\Leftrightarrow\quad\left|\rho\left(T_{w_{0}}^{j}\left(\sigma_{0}\right)\right)\right|=1,\label{eq:fiber-to-length}
\end{equation}
from which \emph{(2)}$\Leftrightarrow$\emph{(3)$\Leftarrow$(4)}
would follow.

 If one defines $q_{s}=\left|C_{s}\sigma_{0}\right|$
for $s\in S$, then whenever $w=s_{1}\cdots s_{\ell\left(w\right)}$
is a reduced word, the fiber $C_{w}\sigma_{0}$ of $\rho$ over $w\sigma_{0}$
is of size $q_{s_{1}}\cdot\ldots\cdot q_{s_{\ell\left(w\right)}}$
(see~\cite[\S6.2]{Garrett1997}). In particular, writing $w_{0}=s_{1}\cdots s_{\ell\left(w_{0}\right)}$,
we have $k=q_{s_{1}}\cdot\ldots\cdot q_{s_{\ell\left(w_{0}\right)}}$.
For each $1\leq i\leq\ell\left(w_{0}\right)$, define a branching
operator $T_{i}$ on $\mathcal{B}\left(d\right)$ by $T_{i}\left(g\sigma_{0}\right)=gC_{s_{i}}\sigma_{0}$,
and observe that $T_{w_{0}}=T_{\ell\left(w_{0}\right)}\circ\ldots\circ T_{1}$.
Since each $T_{i}$ is $q_{s_{i}}$-regular, $T_{w_{0}}^{j}\left(\sigma_{0}\right)$
is of size $k^{j}$ if and only if 
\begin{equation}
\Upsilon:=\left|T_{i}\left(T_{i-1}\left(\ldots\left(T_{1}\left(T_{w_{0}}^{r}\left(\sigma_{0}\right)\right)\right)\right)\right)\right|=k^{r}q_{s_{1}}\ldots q_{s_{i}}\label{eq:Tj_size}
\end{equation}
for every $1\leq r<j$ and $1\leq i\leq\ell\left(w_{0}\right)$. Recall
the Bruhat relations: 
\begin{equation}
C_{w}C_{s}=\begin{cases}
C_{ws} & \ell\left(ws\right)=\ell\left(w\right)+1\\
C_{ws}\cup C_{w} & \ell\left(ws\right)=\ell\left(w\right)-1
\end{cases}\,,\label{eq:Bruhat-relation}
\end{equation}
which hold for every $w\in\mathcal{W}$ and $s\in S$. If $\ell\bigl(w_{0}^{j}\bigr)=j\cdot\ell(w_{0})$
then $\left(s_{1}\cdots s_{\ell\left(w_{0}\right)}\right)^{j}$ is
a reduced word for $w_{0}^{j}$, so that~\eqref{eq:Bruhat-relation}
implies
\[
T_{w_{0}}^{j}\left(\sigma_{0}\right)=\left(C_{w_{0}}\right)^{j}\sigma_{0}=\left(C_{s_{1}}\ldots C_{s_{\ell\left(w_{0}\right)}}\right)^{j}\sigma_{0}=C_{w_{0}^{j}}\sigma_{0},
\]
from which one infers that $\rho\bigl(T_{w_{0}}^{j}\left(\sigma_{0}\right)\bigr)=\bigl\{ w_{0}^{j}\sigma_{0}\bigr\} $
and $\bigl|T_{w_{0}}^{j}\left(\sigma_{0}\right)\bigr|=k^{j}$.

In the other direction, if $\ell\bigl(w_{0}^{j}\bigr)<j\cdot\ell\left(w_{0}\right)$,
consider the first $r,i$ for which $w_{0}^{r}s_{1}\cdots s_{i}$
is not reduced, and let $w'=w_{0}^{r}s_{1}\cdots s_{i-1}$. By the
exchange property of Coxeter groups, one can remove some letter $s$
from $w'$, obtaining a word $w''$ such that $w''s_{i}$ is a reduced
word for $w'$. This also implies that $w''$ is a reduced word for
$w's_{i}$, so that from~\eqref{eq:Bruhat-relation} one has
\[
C_{w_{0}^{r}s_{1}\cdots s_{i-1}}C_{s_{i}}\sigma_{0}=C_{w''}\sigma_{0}\cup C_{w''s_{i}}\sigma_{0}
\]
(a disjoint union, as $\mathcal{W}$ acts freely on $\mathfrak{A}$). 
From this one sees that~\eqref{eq:Tj_size} fails: 
\begin{align*}
\Upsilon & =\left|C_{w_{0}^{r}s_{1}\cdots s_{i-1}}C_{s_{i}}\sigma_{0}\right|=\left|C_{w''}\sigma_{0}\right|+\left|C_{w''s_{i}}\sigma_{0}\right|=\frac{k^{r}q_{s_{1}}\ldots q_{s_{i-1}}}{q_{s}}+\frac{k^{r}q_{s_{1}}\ldots q_{s_{i}}}{q_{s}}\\
 & =k^{r}q_{s_{1}}\ldots q_{s_{i-1}}\left(\frac{1+q_{s_{i}}}{q_{s}}\right)<k^{r}q_{s_{1}}\ldots q_{s_{i}}\,,
\end{align*}
since the building  associated with a simple group is thick ($q_{s}\geq2$). This implies that $|T_{w_0}^{r+1}(\sigma_0)|\leq k^{r+1}$ and $|T_{w_0}^j(\sigma_0)|<k^j$.
Furthermore,  $\bigl|\rho\bigl(T_{w_{0}}^{j}\left(\sigma_{0}\right)\bigr)\bigr|\geq2$,
since 
\[
T_{w_{0}}^{j}\left(\sigma_{0}\right)=C_{w_{0}^{r}s_{1}\cdots s_{i-1}}C_{s_{i}}C_{s_{i+1}}\cdots C_{s_{\ell\left(w_{0}\right)}}\left(C_{s_{1}}\cdots C_{s_{\ell\left(w_{0}\right)}}\right)^{j-r-1}\left(\sigma_{0}\right)
\]
contains the cells $w''s_{i+1}\ldots s_{\ell\left(w_{0}\right)}w_{0}^{j-r-1}\sigma_{0}$
and $w''s_{i}\ldots s_{\ell\left(w_{0}\right)}w_{0}^{j-r-1}\sigma_{0}$ (which are different since $\cW$ acts simply transitively on the chambers of $\mathfrak{A}$), yielding~\eqref{eq:fiber-to-length}.

Next, we recall the fact (cf., e.g.,~\cite[\S12.2]{Garrett1997}) that $\ell\left(w\right)$ (for any $w\in\mathcal{W}$)
equals the number of walls in $\mathfrak{A}$ separating $\sigma_{0}$
from $w\sigma_{0}$. In particular, this applies to $w=w_{0}^{\mu}$,
which acts on $\mathfrak{A}$ by translation, hence for every $j\in\mathbb{N}$
there are $j\cdot\ell\left(w_{0}^{\mu}\right)$ walls separating $\sigma_{0}$
from $w_{0}^{\mu j}\sigma_{0}$. If \emph{(3)} holds, then 
$\ell\bigl(w_{0}^{\mu j}\bigr)=j\cdot\ell\left(w_{0}^{\mu}\right)=j\mu\ell\left(w_{0}\right)$.
But this means that $\left(s_{1}\cdots s_{\ell\left(w_{0}\right)}\right)^{j\mu}$
is reduced for every $j$, hence also $\ell\bigl(w_{0}^{j}\bigr)=j\ell\left(w_{0}\right)$
for every $j$, thus by~\eqref{eq:fiber-to-length} we have \emph{(3)$\Rightarrow$(4)}. Finally,
it is now clear that \emph{(2),(3),(4)} imply \emph{(1)}: for $j\neq j'$,
\emph{(4)} implies that $T_{w_{0}}^{j}\left(\sigma_{0}\right)\cap T_{w_{0}}^{j'}\left(\sigma_{0}\right)=\varnothing$,
and for each $j$, \emph{(4)} and~\eqref{eq:fiber-to-length} together
imply that $\bigl|T_{w_{0}}^{j}\left(\sigma_{0}\right)\bigr|=k^{j}$.
\end{proof} We remark that we can apply this theorem to the $d$-dimensional
geodesic flow $T$ on $\PGL_{d+1}$, even though it is a non-simple
group; one can look at $T^{d+1}$, which belongs to the simple group
$\PSL_{d+1}$ and satisfies the conditions of the theorem.

\section{\label{sec:Zeta-functions}Zeta functions of digraphs and Ramanujan
complexes}

Recall  that Ihara~\cite{ihara1966discrete} associated with
every graph $\cG$ a zeta function defined by
$
Z_{\cG}\left(u\right)=\prod_{\left[\gamma\right]}\frac{1}{1-u^{\ell\left(\gamma\right)}}\,,\label{eq:Ihara-zeta-euler}
$
where $\left[\gamma\right]$ runs over the equivalence classes of
geodesic tailless primitive cycles $\gamma$ in $\cG$ (two cycles are
equivalent if they differ by a cyclic rotation), and $\ell\left(\gamma\right)$
is the length of $\gamma$. Taking a logarithmic derivative, one
gets
\[
Z_{\cG}\left(u\right)=\exp\biggl(\sum_{m=1}^{\infty}\frac{N_{m}\left(\cG\right)}{m}u^{m}\biggr)\,,
\]
where $N_{m}\left(\cG\right)$ is the number of geodesic tailless cycles
in $\cG$ of length $m$. For a $k$-regular graph $\cG$ with $k=q+1$, Ihara showed
that
\begin{equation}
Z_{\cG}\left(u\right)=\left[{\left(1-u^{2}\right)^{\beta_{1}}\det\left(I-uA_{\cG}+u^{2}qI\right)}\right]^{-1}\,,\label{eq:Ihara-det}
\end{equation}
where $\beta_{1}$ is the first Betti number  of $\cG$. From~\eqref{eq:Ihara-det} one deduces (cf.~\cite[p.~59]{Lub10}) that
$\cG$ is Ramanujan iff the associated zeta
function satisfies the Riemann Hypothesis:
\[
Z_{\cG}\left(q^{-s}\right)=\infty\quad\Rightarrow\quad\text{either }q^{-s}\in\bigl\{ \pm1,\pm\tfrac{1}{q}\bigr\} \text{ or }\Re\left(s\right)=\tfrac{1}{2}\,,
\]
i.e., all the poles of $Z_{\cG}\left(u\right)$ are obtained at $u= \pm1,\pm\frac{1}{q}$, or for $|u|=\sqrt{q}$.

Hashimoto~\cite{hashimoto1989zeta} formulated a variant of this by
showing
\[
Z_{\cG}\left(u\right)=\det\left(I-uT\right)^{-1}\,,
\]
where $T$ is the non-backtracking operator on directed edges (coinciding with the case $m=2$, $j=1$ of our geodesic flow in \S\ref{sub:Geodesic-flow}). This
approach generalizes to arbitrary digraphs: indeed, defining
the zeta function of a finite digraph $\cY$ to be
\begin{equation}
Z_{\cY}\left(u\right)=\prod_{\left[\gamma\right]}\frac{1}{1-u^{\ell\left(\gamma\right)}}\,,\label{eq:digraph-zeta}
\end{equation}
where $\left[\gamma\right]$ runs over  the equivalence classes of primitive
(directed) cycles in $\cY$, one has:
\begin{theorem}[\cite{bowen1970zeta,kotani2000zeta}]
\label{thm:digraph-zeta-polynom}If $N_{m}\left(\cY\right)$ is the
number of (directed) cycles of length $m$ in $\cY$, then 
\[
Z_{\cY}\left(u\right)=\exp\biggl(\sum_{m=1}^{\infty}\frac{N_{m}\left(\cY\right)}{m}u^{m}\bigg)=\frac{1}{\det\left(I-uA_{\cY}\right)}\,.
\]
\end{theorem}
Various authors have suggested various zeta functions associated to
simplicial complexes with the hope to generalize the Hashimoto work~\cite{deitmar2006ihara,kang2010zeta,kang2014zeta,fang2013zeta,kang2015zeta,kang2016riemann}.
We will associate a zeta function of this kind with every $G$-equivariant branching operator:
\begin{definition}
For every $G$-equivariant branching operator $T$ on $\fC\subseteq\mathcal{B}=\mathcal{B}(G)$,
the \emph{$T$-zeta function} of a quotient complex
$X=\Gamma\backslash\mathcal{B}$ (as in \S\ref{sec:Ramanujan-complexes})
is defined as 
$\mathfrak{z}_{T}\left(X,u\right)=Z_{\cY_{T,\fX}}(u)$, namely \eqref{eq:digraph-zeta}, with $\gamma$ running over the equivalence classes of primitive
cycles in the digraph defined by $T$ on $\fX$.
\end{definition}
Theorem~\ref{thm:digraph-zeta-polynom} implies that, for $\cY=\cY_{T,\fX}$, one has
\[
\mathfrak{z}_{T}\left(X,u\right)=\det\left(I-uA_\cY\right)^{-1}\,.
\]
We have shown in~\S\ref{sec:Ramanujan-complexes} that if $X=\Gamma\backslash\cB$ is a Ramanujan complex
and $T$ is a $G$-equivariant, $k$-regular, collision-free branching operator on $\cB$, then $\cY_{T,\fX}$ is Ramanujan. In particular, the $T$-zeta function of $X$ satisfies the R.H.; that is, if 
$\mathfrak{z}_{T}\left(X,k^{-s}\right)=\infty$ then either $|s|=1$ or $\Re\left(s\right)\le\frac{1}{2}$, establishing Corollary~\ref{cor:R-H-for-Ramanujan}.
\begin{remark}
 The bound $\left|\Re (s)\right|\leq\frac{1}{2}$ is the true situation: there are poles 
with $\left|\Re (s)\right|<\frac{1}{2}$. In fact, already in the graph case there are poles with $|\Re (s)|=0$, 
and in higher dimension there are also poles with $0<|\Re( s)|<\frac{1}{2}$ (see, e.g.,~\cite{kang2010zeta}).
\end{remark}
\begin{remark} 
Zeta functions of this kind have been studied by a 
number of authors~\cite{deitmar2006ihara,kang2010zeta,kang2014zeta,kang2015zeta,fang2013zeta,kang2016riemann}.
The most general results regarding the R.H.\ are due to Kang~\cite{kang2016riemann},
who studied the case of buildings of type $\widetilde{A}_{n}$
and gave more detailed information about the poles of the zeta functions corresponding to geodesic flows,
via case by case analysis of the representations of $\mathrm{GL}_{n}\left(F\right)$. In~\cite{fang2013zeta}, 
the same is done for the building of type $\widetilde{C}_2$ associated with $\mathrm{PSp}(4)$.
Our simple combinatorial treatment gives
the upper bound on $\Re\left(s\right)$ but not the exact possible
values of it. However, our method applies for any building, and not only those of type $\widetilde{A}_{n}$ and $\widetilde{C}_{2}$.
It seems, in any case, that for combinatorial applications
the upper bound on $\Re\left(s\right)$ suffices, as is illustrated in this paper.
\end{remark}
\begin{remark}
 For $k$-regular graphs,  it is known that the converse also holds: the R.H.\ for the zeta function of a graph implies
it is Ramanujan. In~\cite{kang2010zeta} it was shown that the R.H.\ for zeta functions of geodesic
flows on the building of type $\widetilde{A}_2$ implies that a quotient complex is Ramanujan. Recently, an analogous result was
established by Kamber for buildings of general type~\cite{Kamber2017LP}.
\end{remark}

\subsection*{Acknowledgment} The authors wish to thank Peter Sarnak for many helpful discussions.
E.L.\ was supported in part by NSF grant DMS-1513403 and BSF grant 2014361. A.L.\ was supported by the ERC, BSF and NSF.

\bibliographystyle{abbrv}
\bibliography{mybib}

\end{document}